\documentclass{amsart}

\usepackage{endnotes}
\let\footnote=\endnote

\usepackage[numbers,sort&compress]{natbib}
\usepackage{subcaption}
\usepackage{multirow}
\usepackage{scalerel}
\usepackage{array}
\usepackage{eqnarray}
\usepackage{amssymb}
\usepackage{mathrsfs}
\usepackage{booktabs,caption}
\usepackage[flushleft]{threeparttable}
\usepackage{mathtools}
\usepackage{algorithm}
\usepackage{algorithmic}
\usepackage{latexsym}
\usepackage{bm}
\usepackage{graphicx}
\usepackage{xcolor}
\usepackage{longtable}
\usepackage{textcomp}
\usepackage{tabu}
\usepackage{enumitem}
\usepackage{tikz}
\usepackage{subcaption}
\usetikzlibrary{arrows.meta,positioning}

\def \b {}

\newcommand{\lip}{\left<}
\newcommand{\rip}{\right>}
\newcommand{\lvt}{\left[}
\newcommand{\rvt}{\right]}

\newcommand{\ddd}{,\ldots,}

\newcommand{\iq}{\mbox{\rm IQ}}

\newcommand{\bp}{{ p}}
\newcommand{\bq}{{ q}}

\newcommand{\re}{\mathbb{R}}

\newcommand{\bE}{\mathbb{E}}
\newcommand{\N}{\mathbb{N}}

\newcommand{\Dt}{\Delta}

\newcommand{\st}{{s.t.}}

\newcommand{\lmd}{\lambda}

\newcommand{\mmc}[1]{\mathcal{#1}}

\def\rank{\mbox{rank}}

\newcommand{\bdes}{\begin{description}}
\newcommand{\edes}{\end{description}}

\newcommand{\bal}{\begin{align}}
\newcommand{\eal}{\end{align}}

\newcommand{\bnum}{\begin{enumerate}}
\newcommand{\enum}{\end{enumerate}}

\newcommand{\bit}{\begin{itemize}}
\newcommand{\eit}{\end{itemize}}

\newcommand{\bea}{\begin{eqnarray}}
\newcommand{\eea}{\end{eqnarray}}
\newcommand{\ben}{\begin{equation}}
\newcommand{\een}{\end{equation}}

\newcommand{\baray}{\begin{array}}
\newcommand{\earay}{\end{array}}

\newcommand{\bsry}{\begin{subarray}}
\newcommand{\esry}{\end{subarray}}

\newcommand{\bca}{\begin{cases}}
\newcommand{\eca}{\end{cases}}

\newcommand{\bcen}{\begin{center}}
\newcommand{\ecen}{\end{center}}

\newcommand{\bbm}{\begin{bmatrix}}
\newcommand{\ebm}{\end{bmatrix}}

\newcommand{\btab}{\begin{tabular}}
\newcommand{\etab}{\end{tabular}}
\newcommand{\bx}{{ x}}
\newcommand{\bs}{{ s}}
\newcommand{\by}{{ y}}
\newcommand{\bz}{{z}}

\newcommand{\rd}[1]{{\color{black}#1}}

\newcommand{\mcal}[1]{\mathcal{#1}}

\newcommand{\tp}{\mathsf T}
\newcommand{\copge}{\succcurlyeq_{co}}
\newcommand{\cpge}{\succcurlyeq_{cp}}

\usepackage{amsthm}
\usepackage{url}
\newcommand{\be}{\begin{equation}}
\newcommand{\ee}{\end{equation}}
\numberwithin{equation}{section}
\theoremstyle{plain}
\newtheorem{theorem}{Theorem}[section]
\newtheorem{lemma}[theorem]{Lemma}
\newtheorem{proposition}[theorem]{Proposition}
\newtheorem{corollary}[theorem]{Corollary}
\newtheorem{assumption}[theorem]{Assumption}

\theoremstyle{definition}

\newtheorem{example}[theorem]{Example}
\theoremstyle{remark}
\newtheorem{remark}[theorem]{Remark}

\newsavebox{\arxivmathbox}
\newcommand{\arxivfitmath}[1]{%
  \sbox{\arxivmathbox}{$\displaystyle #1$}%
  \ifdim\wd\arxivmathbox>\linewidth
    \resizebox{\linewidth}{!}{\usebox{\arxivmathbox}}%
  \else\usebox{\arxivmathbox}\fi}

\begin{document}

\title[Max--Min Bilinear Completely Positive Programs]{Max--Min Bilinear Completely Positive Programs: A Semidefinite Relaxation with Tightness Guarantees}
\author[Sarah Yini Gao]{Sarah Yini Gao}
\address{Sarah Yini Gao,
Lee Kong Chian School of Business, Singapore Management University.}
\email{yngao@smu.edu.sg}

\author[Xindong~Tang]{Xindong Tang}
\address{Xindong Tang, Department of Mathematics,
Hong Kong Baptist University,
Kowloon Tong, Kowloon, Hong Kong;
Institute for Research and Continuing Education, Hong Kong Baptist University, Shenzhen, China.}
\email{xdtang@hkbu.edu.hk}

\author[Yancheng Yuan]{Yancheng Yuan}
\address{Yancheng Yuan, Department of Applied Mathematics, The Hong Kong Polytechnic University, Hung Hom, Kowloon, Hong Kong.}
\email{yancheng.yuan@polyu.edu.hk}

\date{}

\begin{abstract}
We study max-min bilinear completely positive (CP) programs and their associated min-max programs.
We establish reformulations of these problems as linear optimization over the  product of the CP cone and the copositive (COP) cone, namely, the COP-CP programs.
A hierarchy of coupled Moment-SOS type semidefinite relaxations is proposed to solve these COP-CP programs.
We show that this approach has asymptotic convergence under mild assumptions. 
Moreover, we use the coupled flat truncation conditions (FTCs) to certify tightness of semidefinite relaxations, and we give sufficient conditions for finite convergence of the hierarchy, under which the FTCs hold at every optimal solution of all sufficiently high-order relaxations.
Numerical experiments are presented to show the effectiveness of our method.
\end{abstract}

\maketitle

\section{Introduction}\label{sec:intro}
For the positive integer $n$, let \(\re^n\) be the \(n\)-dimensional Euclidean space, \(\re_+^n\) be the nonnegative orthant in \(\re^n\), and let \(\mathbb{S}^n\) be the set of \(n\)-by-\(n\) real symmetric matrices. %
A matrix $M\in \mathbb{S}^{n}$ is said to be  {\it completely positive} (CP), denoted by $M\cpge 0$, if there exist vectors $\b{x}^{(1)}\ddd \b{x}^{(r)} \in \re_+^{n}$ such that %
\[M = \b{x}^{(1)}{\b{x}^{(1)}}^{\tp} + \cdots + \b{x}^{(r)}{\b{x}^{(r)}}^{\tp}.\]  
Let $X\in\mathbb{S}^n$ and $Y\in\mathbb{S}^m$ be matrix variables, and let $\mcal{Q}:\mathbb{S}^n\mapsto \mathbb{S}^m $ be a linear operator.
Consider the max-min bilinear CP programs:
\begin{eqnarray}\label{eqn:conic}
\begin{array}{lll}
  \mcal{Z}^*_{\max\min} =  \max\limits_{X \in  \mcal K_{cp}^x} \min\limits_{Y \in  \mcal K_{cp}^y} \b \lip \mcal{Q} (X), Y \rip.
   \end{array}
\end{eqnarray}
In the above, feasible sets
$\mathcal K_{cp}^x \subseteq \mathbb{S}^n$ and $\mathcal K_{cp}^y \subseteq \mathbb{S}^m$ are given by
\begin{equation}\label{eqn:cpconex}
\begin{gathered}
\begin{array}{cc}
\mathcal K_{cp}^x = \left\{\begin{array}{cc} X \in\mathbb{S}^n & \left| \begin{array}{ll}
\lip A_i, X \rip = {b}_i, &\quad i =  1,..., I\\
X \succcurlyeq_{cp} 0 \end{array}\right.
\end{array}\right\}
\end{array},\\
\begin{array}{rll}
\mathcal K_{cp}^y = \left\{\begin{array}{ll} Y \in \mathbb{S}^m & \left| \begin{array}{ll}
\lip B_j, Y \rip = {c}_j, &\quad j = 1,...,J \\
Y\succcurlyeq_{cp} 0
 \end{array}\right.
\end{array}\right\}
\end{array}.  
\end{gathered}
\end{equation}
Throughout the context, we assume that
$K^x_{cp}$ and $K^y_{cp}$ are nonempty. Typically, one is also interested in the associated min-max bilinear CP programs:
\begin{eqnarray}\label{eqn:conic_minmax}
\begin{array}{lll}
   \mcal{Z}_{\min\max}^* =   \min\limits_{Y \in  \mcal K_{cp}^y} \max\limits_{X \in  \mcal K_{cp}^x} \lip \mcal{Q} (X), Y \rip.
    \end{array}
\end{eqnarray}

This class of problems can be viewed as a natural extension of CP programs, where the linear objective over a CP-constrained variable is replaced by a bilinear interaction between two CP-constrained variables.
CP programs and their dual problems, namely, {\it copositive} (COP) programs, have drawn significant attention in recent years and have been shown to be powerful reformulations for many difficult problems; see \citep{preisig1996copositivity,de2002approximation,dukanovic2010copositive,povh2009copositive,povh2007copositive}. 
Particularly, \citep{burer2009copositive} showed that a linearly constrained nonconvex quadratic program with mixed 0-1 variables admits an equivalent CP program reformulation, and \citep{natarajan2011mixed} demonstrated that deriving the first-moment bound for a mixed 0-1 linear optimization problem with uncertain objective is equivalent to solving a CP program. 
Besides that, applications of CP and COP programs include those studied in \citep{xu2025improved,gao2019disruption,hanasusanto2018conic,yan2018design,kong2013scheduling,zhang2022schedule}.
Comprehensive discussions on COP and CP programs appear in \citep{berman2003completely, dur2010copositive, burer2012copositive}.
Various approximation methods have been proposed for COP and CP programs; see \citep{nie2018complete, fan2017semidefinite,de2002approximation}.

Furthermore, max-min bilinear CP programs naturally arise when two decision-makers with opposing objectives interact through a bilinear payoff and each player’s
feasible region admits a CP representation.
Such structures appear in various applications involving strategic interactions and adversarial decision-making. For example, bi-quadratic games, where players' payoff functions depend quadratically on their strategies, can be formulated as max-min bilinear CP programs. We refer to Section~5 for more details.

The CP and COP matrices are closely related to moments and nonnegative polynomials. This connection motivates semidefinite approximation schemes based on the Moment-Sum-of-Squares (SOS) hierarchy \cite{lasserre2001global}, which has become a central tool in polynomial and moment optimization. 
Moment-SOS methods have been applied to a
wide range of problems.
We refer to \citep{henrion2009approximate,nie2023convex,nie2023rational,li2024geometry,pauwels2017positivity,magron2023sparse,zhong2024towards,huang2026finite} for some recent applications. 
More closely related to our work, \cite{laraki2012semidefinite} developed semidefinite relaxations for min–max optimization and mixed-strategy polynomial games.
Comprehensive surveys and textbooks for the Moment-SOS hierarchy may be found in \citep{lasserre2009moments,nie2023moment,blekherman2012semidefinite,lasserre2024moment}.

\medskip
\noindent{\bf Contributions }
We study max-min bilinear CP programs and their associated min-max programs.
We establish equivalent reformulations of these problems as linear optimization over the Cartesian product of the CP cone and the COP cone, which we call {\it COP-CP programs}.
A key difficulty of solving COP-CP programs is that the CP and COP cones are usually approximated from opposite sides, so combining these approximations does not directly yield a one-sided relaxation. We address this difficulty through a coupled moment–SOS hierarchy. Our main contributions are as follows.

\begin{itemize}
\item We formulate the max–min and min–max bilinear CP problems as a pair of conic dual COP-CP programs.
We show that such reformulations are equivalent to original problems under suitable strict-feasibility conditions.

\item We develop a hierarchy of coupled semidefinite relaxations for solving COP-CP programs.
We show that this approach has asymptotic convergence under some strict feasibility and attainment assumptions.  

\item We use coupled {\it flat truncation conditions} (FTCs) to certify the tightness of the semidefinite relaxations numerically.
When the flat-truncation conditions hold at a given relaxation order and semidefinite relaxations have no duality gap, we show that the optimal solutions to the max–min and min–max bilinear CP problems can be recovered from the solutions to the semidefinite relaxations.

\item We establish sufficient conditions for finite convergence of the hierarchy of coupled semidefinite relaxations. Specifically, we prove that the FTCs hold at every optimal solution of all sufficiently high-order relaxations, which guarantees tightness at all such orders.

\item For the application side, we consider the zero-sum bi-quadratic games with mixed-binary strategy sets and formulate max–min bilinear CP programs to solve such bi-quadratic games.
Numerical experiments on {\it cyclic Colonel Blotto games} illustrate the effectiveness of our proposed method for solving such application problems.

\end{itemize}

The remainder of this paper is organized as follows. In \S2, we review preliminaries on COP/CP matrices and polynomial optimization. 
In \S3, we introduce COP-CP program reformulations for max--min and min--max bilinear CP programs, and propose the hierarchy of coupled semidefinite relaxations.
Asymptotic and finite convergence results for this hierarchy are established in \S4.
In \S5, we study zero-sum bi-quadratic games and their max--min bilinear CP reformulations.
Numerical experiments of the proposed semidefinite relaxation method are presented in \S6.
Conclusions are made in \S7.
Proofs of the main results in Sections 3--5 are provided in Appendices A--C.

\section{Preliminaries}\label{sc:pre}
\subsection*{Notation} 
For a given cone $\mcal{K}$, its dual cone is denoted as $\mcal{K}^{\star}$.  The operator $(\cdot)^\tp$ denotes the transpose operator.
The operator $\operatorname{vec}(\cdot)$ denotes matrix vectorization, and $\lip A, B \rip$ denotes the inner product between matrices $A$ and $B$. Inequalities between vectors are interpreted componentwise, and \(e\) denotes the all-ones vector of appropriate dimension.
For a vector $u$, $\Lambda(u)$ represents the diagonal matrix given by $u$.

\subsection{Copositive and Completely Positive Matrices}
A matrix $W\in \mathbb{S}^{n}$ is {\it copositive} (COP), denoted as $W\succcurlyeq_{co}0$, if
\[ \b{x}^{\tp} W \b{x} \ge 0,\quad \mbox{for all}\ x\in \re^{n}_+. \]
An equivalent characterization of copositivity can be given in terms of the simplex set, $\Delta^{n}$ in $\re^{n}$,
\be\label{eq:simplex_n} \Delta^{n}:=\{\b{x}\in\re^{n}_+ \mid  \b{e}^{\tp}\b{x}= 1,\ 1-\b{x}^{\tp}\b{x} \ge 0\}. \ee
Specifically, $W$ is COP if and only if $\b{x}^{\tp} W \b{x}\ge 0$ for all $\b{x}\in \Delta^{n}$. 
We remark that though $ 1-\b{x}^{\tp}\b{x} \ge 0$ in (\ref{eq:simplex_n}) is not necessary to define the simplex set, it facilitates the construction of tighter semidefinite relaxations.
Furthermore, we say $W$ is {\it strictly copositive} if $\b{x}^{\tp} W \b{x} > 0$ for all $\bx\in\Delta^n$, denoted as $W\succ_{co} 0$.

The dual cone of the COP matrix cone is the cone of CP matrices.
Corresponding to the definition of a COP matrix on the simplex $\Delta^{n}$ given in  (\ref{eq:simplex_n}), a matrix $X\in \mathbb{S}^{n}$ is said to be CP, denoted by $X\cpge 0$, if there exist vectors $\b{x}^{(1)}\ddd \b{x}^{(r)} \in \Delta^{n}$ and positive scalars $\lmd_1\ddd \lmd_r$ and such that 
\[X = \lmd_1 \cdot \b{x}^{(1)}{\b{x}^{(1)}}^{\tp} + \cdots + \rd{\lmd_r} \cdot \b{x}^{(r)}{\b{x}^{(r)}}^{\tp}.\]  

\subsection{Truncated Multi-Sequences and Nonnegative Quadratic Forms}
\label{sc:tms}

We now introduce the concepts of truncated multi-sequences (tms) and nonnegative quadratic forms, which play a key role in our subsequent analysis. %
Let \(n \in \mathbb{N}\) be the dimension and \(d \in \mathbb{N}\) the degree. Define
\[
\N^{n}_{d,\mathrm{\hom}} \;=\; \bigl\{ \b{\alpha} = (\alpha_1, \ldots, \alpha_n) \in \N^n 
\;\big|\;
\alpha_1 + \cdots + \alpha_n = d 
\bigr\},\quad
\N^{n}_d = \bigcup\limits_{i=0}^d \N^{n}_{i,\mathrm{\hom}}.
\]
A {\it truncated multi-sequence} (tms) of degree $d$ is a sequence labelled by integer tuples in $\N^{n}_d$,
and a {\it homogeneous truncated multi-sequence} (htms) of degree $d$ is a sequence labelled by integer tuples in $\N^{n}_{d,hom}$.
For each $\b{\alpha}=(\alpha_1\ddd \alpha_{n})\in\N^{n}_d$ and the vector $\b{x}\in\Dt^n$,
define
${\b x}^{{\b\alpha}}:=x_1^{\alpha_1}x_2^{\alpha_2}\dots x_{n}^{\alpha_{n}}.$
The tms and htms of degree $d$ generated by $\b{x}$ are then given by 
\[[\b{x}]_d:=(\b x^{\b{\alpha}})_{\b{\alpha}\in \N^{n}_d} \in \re^{\N^{n}_d},\quad [\b{x}]^{\hom}_d:=({\b x}^{\b{\alpha}})_{\b{\alpha}\in \N^{n}_{d,hom}} \in \re^{\N^{n}_{d,\hom}}.\]
For a tms $\b{\xi}$, we can also label its component by $\xi_{\b{\alpha}}$ or explicitly as $\xi_{\alpha_1\alpha_2\dots\alpha_n}$, which takes the value ${x}^{\b{\alpha}}$ if $\b{\xi}$ is generated by $\b{x}\in\Dt^n$.

Consider the cone of degree $2$ htms supported on $\Delta^{n}$. Denoted by
\begin{equation}
\label{eq:R2hom}
\mcal{R}[\Delta^n]_2^{\mathrm{\hom}}
\;:=\;
\Bigl\{
\lambda_1 [\b{x}^{(1)}]_2^{\mathrm{\hom}} 
+\cdots+ 
\lambda_r [\b{x}^{(r)}]_2^{\mathrm{\hom}}
\;\Big|\; 
r \in \mathbb{N},\ 
\lambda_i > 0,\ 
\b{x}^{(i)} \in \Delta^n
\Bigr\}.
\end{equation}
Then $\mmc{R}[\Delta^{n}]^{\hom}_2$ is the conic hull of the htms $[\b{x}]_2^{\hom}$ for $\b{x}\in\Dt^n$.
Its dual cone is the cone of {\it quadratic forms} that are nonnegative on $\Dt^n$, that
\begin{equation}\label{eq:P2hom}
\mmc{P}[\Delta^{n}]^{\hom}_2:=\{p\in\re[x]^{\hom}_2\mid p(\b{x})\ge0,\ \forall \b{x}\in\Delta^{n}\},\end{equation} where $\re[{x}]^{\hom}_2$ denotes the space of homogeneous polynomials of degree $2$ with coefficients in $\re$.

\subsection{Moment and Sum-of-Squares Relaxations}\label{sc:pre:relx}
We introduce semidefinite relaxations for $\mmc{P}[\Delta^{n}]^{\hom}_2$ and $\mmc{R}[\Delta^{n}]^{\hom}_2$, 
which are useful for solving the max-min bilinear CP programs.

For a variable $\b{x}\in\re^n$ and an integer $k\in\N$, let $\re[\b{x}]_{2k}$ denote the set of polynomials in $\b{x}$ whose degrees are at most $2k$.
A polynomial $\sigma\in\re[\b{x}]_{2k}$ is a SOS if there exist polynomials $\sigma_1\ddd \sigma_l\in\re[\b{x}]_k$ such that
$\sigma = \sigma_1^2+\sigma_2^2+\dots+\sigma_l^2.$
We write $\Sigma[\b{x}]_{2k}$ to denote the set of all SOS polynomials in $\re[\b{x}]_{2k}$. Note that if a polynomial is SOS, then it is nonnegative.

Recall that the standard simplex set $\Delta^{n}$ can be described as in (\ref{eq:simplex_n}).
To approximate $\mmc{P}[\Delta^{n}]^{\hom}_2$, for $k \in \N_+$, we define the $2k$ truncation for the sum of {\it ideal} and {\it quadratic module} generated by the representation (\ref{eq:simplex_n}) for $\Delta^{n}$ as
\begin{equation*}
\resizebox{\linewidth}{!}{$\displaystyle %
\begin{array}{c}
\iq[\Delta^{n}]_{2k}:= \left\{ (1-e^{\tp}\b{x})\cdot q + \sigma_0 + \sum\limits_{i=1}^n x_i \sigma_i + (1-\b{x}^{\tp}\b{x})\sigma_{n+1}\left| \begin{array}{c}
  q\in\re[x]_{2k-1}, \sigma_0\in\Sigma[x]_{2k},\\
  \sigma_1\ddd \sigma_{n+1}\in\Sigma[x]_{2k-2}
\end{array}
\right. \right\}.
\end{array}$}
\end{equation*}
For each $k \in \N_+$, the set $\iq[\Delta^{n}]_{2k}$ is a closed convex cone contained in $\re[x]_{2k}$. 
It is clear that every polynomial $p$ in $\iq[\Delta^{n}]_{2k}$ is nonnegative on $\Dt^{n}$.
So $\iq[\Delta^{n}]_{2k}\cap \re[x]_2^{\hom}$ is an inner approximation for $\mmc{P}[\Delta^{n}]^{\hom}_2$.
The following
Positivstellensatz gives asymptotic tightness for strictly positive
polynomials.

\begin{lemma}[Putinar's Positivstellensatz \citep{putinar1993positive}]\label{lem:iqandp} 
If $p(x) > 0$ for all $x\in\Dt^n$, then $p\in \iq[\Delta^{n}]_{2k}$ for all sufficiently large $k$.
\end{lemma}
Moreover, we remark that even if a polynomial $p$ is nonnegative but not strictly positive on $\Delta^{n}$, there exist conditions that guarantee $p\in\iq[\Delta^{n}]_{2k}$, see \citep{de2011lasserre,nie2014optimality,lasserre2009convexity,nie2023moment}.

We next consider semidefinite relaxations for the cone of htms $\mcal{R}[\Delta^n]^{\hom}_2$.
For an integer $k\in\N$, let $\b{\xi}\in\re^{\N^{n}_{2k}}$ be a tms.
The {\it Riesz functional} $\mmc{L}_{\b{\xi}}$ is a linear functional that maps polynomials in $\re[\b{x}]_{2k}$ to $\re$ such that
\[\mmc{L}_{\xi}(x^{\b{\alpha}}) = \xi_{\b{\alpha}},\quad \forall\, \b{\alpha}\in\N^{n}_{2k}.\] 
Then, for each $\b{\xi}$ and polynomial expressions in $\b{x}$, one can construct the {\it moment matrix} and various {\it localizing matrices}. In particular, define
\be\label{eq:Sk} 
\b{\xi} \in \mathscr{S}[\Dt^n]_{2k} := \left\{ \b{\xi}\in\re^{\N^{n}_{2k}} \left|
\begin{array}{c}
M_{k}[\b{\xi}]\succeq 0,\  \mathscr{V}^{(2k)}_{1-e^{\tp}x}[\b{\xi}] = 0,\  L^{(k)}_{1-x^{\tp}x}[\b{\xi}] \succeq 0,\\
L^{(k)}_{x_i}[\b{\xi}] \succeq 0,\quad \forall i=1\ddd n
\end{array}
\right.
\right\}.\ee
For vectors and matrices of polynomials, the image of the Riesz functional $\mmc{L}_{\xi}$ is defined entrywise.
For the tms $\xi$, the moment matrix $M_{k}[\b{\xi}]$ and the localizing vectors and matrices are
\[\begin{aligned}
M_{k}[\b{\xi}]:=\mmc{L}_{\b{\xi}}([\b{x}]_{k}[\b{x}]_{k}^{\tp}),\quad & \mathscr{V}^{(2k)}_{1-\b{e}^{\tp}\b{x}}[\b{\xi}]:=\mmc{L}_{\b{\xi}}((1-\b{e}^{\tp}\b{x})\cdot [\b{x}]_{2k-1}),\\
L^{(k)}_{x_i}[\b{\xi}]:=\mmc{L}_{\xi}(x_i\cdot [\b{x}]_{k-1}[\b{x}]_{k-1}^{\tp}),\quad &
L^{(k)}_{1-\b{x}^{\tp}\b{x}}[\b{\xi}]:=\mmc{L}_{\xi}((1-\b{x}^{\tp}\b{x})\cdot [\b{x}]_{k-1}[\b{x}]_{k-1}^{\tp}),\end{aligned}\]
The set $\mathscr{S}[\Delta^n]_{2k}$ is a closed convex cone, and is dual to $\iq[\Delta^{n}]_{2k}$, i.e., $(\iq[\Delta^{n}]_{2k})^{\star} = \mathscr{S}[\Delta^n]_{2k}$. 
The following lemma is straightforward by Proposition~8.2.2 of \cite{nie2023moment}.
\begin{lemma}\label{lem:rands}
    For every $k\in \N_+$, $\mathscr{S}[\Delta^n]_{2k}$ is an outer relaxation of $\mcal{R}[\Delta^n]^{\hom}_2$, in the sense that 
$$\mcal{R}[\Delta^n]^{\hom}_2 \subseteq \{ \b{\xi}|^{\hom}_2 : \b{\xi} \in \mathscr{S}[\Delta^n]_{2k}\}.$$
\end{lemma}

\section{Max-Min Bilinear Optimization over the CP Cones and COP-CP Reformulation}\label{sec:formulation}

In this section, we first introduce a COP-CP program reformulation for max-min and min-max bilinear CP programs.
Then, we propose a coupled semidefinite approximation scheme for solving COP-CP programs, and thus, max-min and min-max bilinear CP programs.
\subsection{A COP-CP Conic Reformulation}\label{sec:conic}
We derive a single-stage conic formulation for the max--min problem \eqref{eqn:conic}
by dualizing the inner CP program. 
For a given $X\in \mcal{K}^x_{cp}$, the inner minimization problem in \eqref{eqn:conic} is
\be\label{eq:inner_min}\left\{
\begin{array}{cl}
\min\limits_{Y \in  \mathbb{S}^m} & \lip \mcal{Q} (X), Y \rip\\
\st & \lip B_j, Y \rip = {c}_j, \quad j = 1,...,J, \\
& Y\succcurlyeq_{cp} 0.
\end{array}\right.
\ee
By introducing dual variables $\b z\in\re^J$ and $Z\in \mathbb{S}^m$ with $Z\succcurlyeq_{co} 0$, the Lagrangian of (\ref{eq:inner_min}) is
\[ L(Y,\bz,Z;X) = \langle \mcal{Q}(X),Y \rangle - \sum_{j=1}^J z_j (\langle B_j,Y \rangle - c_j) - \langle Y,Z \rangle. \]
Thus, the conic dual problem of (\ref{eq:inner_min}) reads
\be\label{eq:inner_min_dual}
\left\{
\begin{array}{cl}
\max\limits_{\b{z}\in\re^J,Z\in \mathbb{S}^m} & \b{c}^{\tp}\b{z} \\
\st & \displaystyle \mcal{Q}(X) - \sum_{j=1}^J z_j B_j = Z,\quad Z\succcurlyeq_{co} 0.
\end{array}\right.
\ee
Replacing the inner minimization in  \eqref{eqn:conic} by its conic dual gives the following COP--CP conic program:
\be\label{eq:lp_max}
\left\{
\begin{array}{ccl}
\mcal{W}_{\max}^*\coloneqq& \max\limits_{X,\b{z},Z} & \b{c}^{\tp}\b{z} \\
& \st & \lip A_i, X \rip = {b}_i, \quad i = 1,..., I,\\
&    & \displaystyle \mcal{Q}(X) - \sum_{j=1}^J z_j B_j = Z,\\
&    & X\succcurlyeq_{cp} 0,\ Z\succcurlyeq_{co} 0,\ X\in\mathbb{S}^n,\ \b{z}\in{\re^J},\ Z\in\mathbb{S}^m.
\end{array}\right.
\ee
Similarly, considering the dual for the inner maximization of the min--max problem \eqref{eqn:conic_minmax}
yields 
\be\label{eq:lp_min}\left\{
\begin{array}{ccl}
\mcal{W}_{\min}^*\coloneqq & \min\limits_{Y,\b{w},W} & \b{b}^{\tp}\b{w} \\
& \st & \lip B_j, Y \rip = {c}_j, \quad j = 1,..., J,\\
&    & \displaystyle \sum_{i=1}^I w_i A_i - \mcal{Q}^{\star}(Y) = W,\\
&    & Y\succcurlyeq_{cp} 0,\ W\succcurlyeq_{co} 0,\ Y\in\mathbb{S}^m,\ {\b{w}\in\re^I},\ W\in\mathbb{S}^n.
\end{array}\right.
\ee
Here, $\mcal{Q}^{\star}$ denotes the adjoint operator of $\mcal{Q}$.
Interestingly, the maximization problem (\ref{eq:lp_max}) and the minimization problem (\ref{eq:lp_min}) are dual to each other, and there exists an ascending chain of optimal values for the above problems, as demonstrated in the following result.

\begin{proposition}\label{prop:maxmin_duality}
The maximization problem (\ref{eq:lp_max}) and the minimization problem (\ref{eq:lp_min}) are dual to each other, and
\[ \mcal{W}_{\max}^* \le \mcal{Z}^*_{\max\min} \le \mcal{Z}^*_{\min\max} \le \mcal{W}_{\min}^*.\]

\end{proposition}
To obtain exact reformulations, we consider the following strict-feasibility
conditions.
\begin{assumption}\label{as:co}
  \begin{enumerate}
    \item[(a)] For the symmetric matrices $B_1\ddd B_J$, there exist $z_1\ddd z_J$ such that $\sum\limits_{j=1}^Jz_jB_j\succ_{co}0$.
    \item[(b)] For the symmetric matrices $A_1\ddd A_I$, there exist $w_1\ddd w_I$ such that $\sum\limits_{i=1}^Iw_iA_i\succ_{co}0$.
  \end{enumerate}
\end{assumption}
Under these conditions, the COP--CP reformulations are exact for max-min/min-max bilinear CP programs. %
\begin{lemma}\label{lem:z=w}
Under Assumption~\ref{as:co}(a), $\mcal{Z}^*_{\max\min} = \mcal{W}^*_{\max}$. Likewise, under Assumption~\ref{as:co}(b), $\mcal{Z}^*_{\min\max} = \mcal{W}^*_{\min}$. 
\end{lemma}
To end this subsection, we remark if at least one $B_j$ for $j\in [J]$ is \emph{strictly copositive}, then Assumption~\ref{as:co}(a) holds. 
Similarly, if $A_i$ is strictly copositive for some $i\in [I]$, then $\mcal{Z}^*_{\min\max} = \mcal{W}_{\min}^*$. 

\subsection{A Semidefinite Approximation Method for COP-CP Programs}\label{sec:relax}
We propose a semidefinite approximation method for solving COP-CP programs \eqref{eq:lp_max} and \eqref{eq:lp_min}. 

Recall that $\mcal{R}[\Delta^n]_2^{\hom}$ is the cone of degree-2 htms supported on $\Delta^n$ given by (\ref{eq:R2hom}). Define a linear map $\varphi_n:\mathbb{S}^{n} \mapsto\re^{\N^{n}_{2,\hom}} $ {such that for all $\lmd_1\ddd \lmd_r\in\re$ and $\b{x}^{(1)}\ddd \b{x}^{(r)} \in \Delta^{n}$, it holds}
\begin{eqnarray*}%
    \varphi_n \big(\sum_{i=1}^{r} \lmd_i \b{x}^{(i)}{\b{x}^{(i)}}^{\tp}\big) = \sum_{i=1}^{r}\lmd_i[\b{x}^{(i)}]_2^{\hom}.
\end{eqnarray*}
Then $\varphi_n$ is an isomorphism.
By the definitions of CP matrices and $\mcal{R}[\Delta^n]_2^{\hom}$, one sees that $X\in \mathbb{S}^{n}$ is CP if and only if $\varphi_n({X}) \in \mcal{R}[\Delta^n]_2^{\hom}$.
Similarly, the cone of COP matrices can be identified with the dual cone of $\mcal{R}[\Delta^n]_2^{\hom}$, i.e., the cone of nonnegative quadratic forms, $\mmc{P}[\Delta^{n}]^{\hom}_2$, given by (\ref{eq:P2hom}).  Let $\psi_n:\mathbb{S}^{n} \mapsto \re[x]_2^{\hom}$ be the linear map such that for any matrix $W\in\mathbb{S}^{n}$, it holds
\begin{eqnarray*}%
    \psi_n(W) = \b{x}^{\tp} W \b{x}.
\end{eqnarray*}
The linear map $\psi_n$ is a linear isomorphism from $\mathbb{S}^n$ to $\re[\bx]_2^{\hom}$, and by definition, $W\copge0$ if and only if $\psi_n(W)\in \mmc{P}[\Delta^{n}]^{\hom}_2$.
More details on these relationships can be found in Chapter~9 of \citep{nie2023moment}.

Define analogous maps $\varphi_m$ and $\psi_m$ in lieu of $\varphi_n$ and $\psi_n$ when we refer to variables on $m$ dimensional space.  
Using these maps, the COP-CP program \eqref{eq:lp_max} is equivalent to
\be\label{eq:lp_max_pop}\left\{
\begin{array}{ccl}
\mcal{W}_{\max}^*\coloneqq& \max\limits_{X,\b{z},Z} & \b{c}^{\tp}\b{z} \\
& \st & \lip A_i, X \rip = {b}_i, \quad i = 1,..., I,\\
&    & \displaystyle \mcal{Q}(X) - \sum_{j=1}^J z_j B_j = Z,\\
&    & \varphi_n(X) \in \mmc{R}[\Delta^{n}]^{\hom}_2,\ \psi_m(Z)\in\mmc{P}[\Delta^{m}]^{\hom}_2.
\end{array}\right.
\ee
For coupled relaxation orders $(k_1,k_2) \in \N_+^2$, 
denote by $\b{\xi}|_{2}^{\hom}$ the homogeneous degree-2 truncation for $\b{\xi}\in \re^{\N^{n}_{2k_1}}$, i.e., $\b{\xi}|_{2}^{\hom}  := ({\xi}_{\b{\alpha}})_{\b{\alpha}\in \N^{n}_{2,\hom}}$. Replacing
$\mathcal R[\Delta^n]^{\rm hom}_2$ by its moment outer approximation, and replacing
$\mathcal P[\Delta^m]^{\rm hom}_2$ by its SOS inner approximation yields
\be\label{eq:lp_max_relx}\left\{
\begin{array}{ccl}
\mcal{W}_{\max}^{(k_1,k_2)}\coloneqq& \max\limits_{X,\b{z},Z,\b{\xi}} & \b{c}^{\tp}\b{z} \\
& \st & \lip A_i, X \rip = {b}_i, \quad i = 1,..., I,\\
&    & \displaystyle \mcal{Q}(X) - \sum_{j=1}^J z_j B_j = Z,\ \varphi_n(X) = \b \xi|^{\hom}_{2},\\
&    & \b\xi \in \mathscr{S}[\Dt^n]_{2k_1},\ \psi_m(Z)\in\iq[\Delta^{m}]_{2k_2}.
\end{array}\right.
\ee
This is a semidefinite program.
{Similarly, for the min-max problem \eqref{eq:lp_min}, one can also derive the following semidefinite relaxation:}
\be\label{eq:lp_min_relx}
\left\{ \begin{array}{ccl}
\mcal{W}_{\min}^{(k_1,k_2)}\coloneqq & \min\limits_{Y,\b{w},W,\b{\nu}} & \b{b}^{\tp}\b{w} \\
& \st & \lip B_j, Y \rip = {c}_j, \quad j = 1,..., J\\
&    & \displaystyle \sum_{i=1}^I w_i A_i - \mcal{Q}^{\star}(Y) = W,\ \varphi_m(Y) = \b{\nu}|^{\hom}_{2}\\
&    & \b{\nu} \in \mathscr{S}[\Dt^m]_{2k_2},\ \psi_n(W)\in\iq[\Delta^{n}]_{2k_1}.
\end{array} \right.
\ee
One may verify that \eqref{eq:lp_min_relx} is the dual of \eqref{eq:lp_max_relx}, so by weak duality, $\mcal{W}_{\max}^{(k_1,k_2)}\le \mcal{W}_{\min}^{(k_1,k_2)}$.
Solving the duality pair of SDPs (\ref{eq:lp_max_relx}) and (\ref{eq:lp_min_relx}) gives a hierarchy of coupled semidefinite relaxation for the COP-CP programs (\ref{eq:lp_max})-(\ref{eq:lp_min}).

We emphasize that these SDPs do not form a standard one-sided hierarchy in an
immediate way:
The SOS cone $\rd{\iq[\Delta^{m}]_{2k_2}}$ approximates the COP cone from the inside for any $k_2 \in \N_+$, 
yet the moment cone $\mathscr{S}[\Delta^n]_{2k_1}$ approximates the CP cone from the outside for any $k_1 \in \N_+$. %
Consequently, it is not immediately clear that whether the relaxation problems (\ref{eq:lp_max_relx}) and (\ref{eq:lp_min_relx}) provide upper or lower bounds for \eqref{eq:lp_max} and \eqref{eq:lp_min}.
The asymptotic and finite convergence properties of this
coupled hierarchy are established in the next section.

\section{Convergence Analysis}\label{sec:conv_analy}
This section establishes the convergence properties for the coupled semidefinite relaxation method.
We say the hierarchy of coupled semidefinite relaxations (\ref{eq:lp_max_relx}) and (\ref{eq:lp_min_relx}) has {\it asymptotic convergence} if 
\[\lim_{\min\{k_1,k_2\}\to \infty}\mcal{W}^{(k_1,k_2)}_{\max} = \mcal{W}^*_{\max} ,\quad  
\lim_{\min\{k_1,k_2\}\to \infty}\mcal{W}^{(k_1,k_2)}_{\min} = \mcal{W}^*_{\min}. \]
At a given relaxation order $(k_1,k_2)$, the coupled semidefinite relaxations are said to be {\it tight} if
\[\mcal{W}^{(k_1,k_2)}_{\max} = \mcal{W}^*_{\max} = \mcal{W}^*_{\min} = \mcal{W}^{(k_1,k_2)}_{\min}. \]
The hierarchy has {\it finite convergence} if the coupled relaxations are tight for all sufficiently large $k_1$ and $k_2$.

\subsection{Asymptotic Convergence}
We first study the asymptotic convergence properties of the semidefinite relaxations \eqref{eq:lp_max_relx}-\eqref{eq:lp_min_relx} relative to the original problems \eqref{eq:lp_max}-\eqref{eq:lp_min}. Throughout the rest of the paper, we use the following notation shorthand 
\[ (X,\b{z},Z)^*  \coloneqq (X^*,\b{z}^*,Z^*),\quad (Y,\b{w},W)^*  \coloneqq (Y^*,\b{w}^*,W^*),\]
where $(X^*,\b{z}^*,Z^*)$ is the maximizer of (\ref{eq:lp_max}) and  $(Y^*,\b{w}^*,W^*)$ is the minimizer of (\ref{eq:lp_min}). %
Also, for any relaxation orders $(k_1,k_2)\in\mathbb N_+^2$, we denote the following for optimal solutions of \eqref{eq:lp_max_relx} and \eqref{eq:lp_min_relx}, respectively:
\[ \begin{aligned}
(X,\b{z},Z,\b{\xi})^{(k_1,k_2)} & \coloneqq (X^{(k_1,k_2)},\b{z}^{(k_1,k_2)},Z^{(k_1,k_2)},\b{\xi}^{(k_1,k_2)}),\\
(Y,\b{w},W,\b{\nu})^{(k_1,k_2)} & \coloneqq (Y^{(k_1,k_2)},\b{w}^{(k_1,k_2)},W^{(k_1,k_2)},\b{\nu}^{(k_1,k_2)}).
\end{aligned}
\]

\begin{theorem}\label{tm:asym_conv}
For the pair of COP-CP programs (\ref{eq:lp_max})-(\ref{eq:lp_min}), suppose Assumption~\ref{as:co} holds.
Consider the semidefinite approximations (\ref{eq:lp_max_relx})-(\ref{eq:lp_min_relx}) with $(k_1,k_2)\in \N^2_+$.
\begin{enumerate}
  \item  If the maximum value of (\ref{eq:lp_max}) is attainable, then (\ref{eq:lp_max_relx}) has a nonempty feasible set when $k_2$ is big enough, and $\mcal{W}^{(k_1,k_2)}_{\max} \to \mcal{W}^*_{\max}$ as $\min\{k_1,k_2\}\to \infty$;
  \item  If the minimum value of (\ref{eq:lp_min}) is attainable, then (\ref{eq:lp_min_relx}) has a nonempty feasible set when $k_1$ is big enough, and $\mcal{W}^{(k_1,k_2)}_{\min} \to \mcal{W}^*_{\min}$ as $\min\{k_1,k_2\}\to \infty$.
\end{enumerate}
\end{theorem}

Combining the above theorem with Proposition \ref{prop:maxmin_duality} and Lemma \ref{lem:z=w}, we have the following: 
\begin{corollary}\label{cor:asym_conv}
    Let $(k_1,k_2)\in \N^2_+$ be relaxation orders for the semidefinite relaxations (\ref{eq:lp_max_relx})-(\ref{eq:lp_min_relx}) approximating (\ref{eq:lp_max})-(\ref{eq:lp_min}). If both the maximum value of (\ref{eq:lp_max}) and the minimum value of (\ref{eq:lp_min}) are attainable, and Assumption~\ref{as:co} holds, then \[
\resizebox{\linewidth}{!}{$\displaystyle \lim_{\min\{k_1,k_2\}\to \infty}\mcal{W}^{(k_1,k_2)}_{\max} = \mcal{W}^*_{\max} = \mcal{Z}^*_{\max\min}\le \mcal{Z}^*_{\min\max} =  \mcal{W}^*_{\min} =  \lim_{\min\{k_1,k_2\}\to \infty}\mcal{W}^{(k_1,k_2)}_{\min}. $}
\]
\end{corollary}

\subsection{Flat Truncation and Finite Convergence}\label{sc:finite_conv}
We next establish flat truncation certificates for tightness and sufficient conditions for finite convergence.
Recall that our SDP constructions approximate the CP and COP cones from the opposite sides, which complicates direct bounding arguments. 
With the flat truncation condition (FTC), however, the semidefinite relaxation effectively reduces to a one-sided approximation. The following result formalizes this observation.

\begin{proposition}\label{prop:ft}
For the coupled relaxation order $(k_1,k_2)$, let $(X,\b{z},Z,\b{\xi})^{(k_1,k_2)}$ be the maximizer of (\ref{eq:lp_max_relx}) and let $(Y,\b{w},W,\b{\nu})^{(k_1,k_2)}$ be the minimizer of (\ref{eq:lp_min_relx}).
\begin{enumerate}
\item  If $\b{\xi}^{(k_1,k_2)}$ satisfies the FTC, i.e., there exists $t\le k_1$ such that \be\label{eq:ft_max} r_1:=\rank\, M_{t}[\b{\xi}^{(k_1,k_2)}]=\rank\, M_{t-1}[\b{\xi}^{(k_1,k_2)}],\ee  
{then $X^{(k_1,k_2)}\succcurlyeq_{cp} 0$} and $\mcal{W}_{\max}^{(k_1,k_2)} \le \mcal{W}_{\max}^*$;
\item  If $\b{\nu}^{(k_1,k_2)}$ satisfies the FTC, i.e., there exists $s\le k_2$ such that \be\label{eq:ft_min} r_2:=\rank\, M_{s}[\b{\nu}^{(k_1,k_2)}]=\rank\, M_{s-1}[\b{\nu}^{(k_1,k_2)}],
\ee 
{then $Y^{(k_1,k_2)}\succcurlyeq_{cp} 0$} and $\mcal{W}_{\min}^{(k_1,k_2)} \ge \mcal{W}_{\min}^*$.
\end{enumerate}
\end{proposition}
\begin{remark}
The FTCs (\ref{eq:ft_max}) and (\ref{eq:ft_min}) give sufficient but not necessary condition for the CPness of $X^{(k_1,k_2)}$ and $Y^{(k_1,k_2)}$.
Even if (\ref{eq:ft_max}) does not hold, one may still check if $X^{(k_1,k_2)}$ is CP or not by solving the truncated moment problem that whether $\varphi_n(X^{(k_1,k_2)}) \in \mmc{R}[\Delta^{n}]^{\hom}_2$. 
This technique is particularly useful when the FTCs (\ref{eq:ft_max})--(\ref{eq:ft_min}) hold with big $k_1$ and $k_2$, but $X^{(k_1,k_2)}$ and $Y^{(k_1,k_2)}$ achieve CPness with lower relaxation orders; see Example~\ref{ex:toy} for such an exposition.
We refer to \cite{nie2014truncated} for more details about truncated moment problems.
\end{remark}

By Proposition \ref{prop:ft}, once FTCs hold, the {semidefinite relaxation} hierarchy becomes one-sided: $\mcal{W}_{\max}^{(k_1,k_2)}$ gives a lower bound for the COP-CP maximization problem (\ref{eq:lp_max}),
while $\mcal{W}_{\min}^{(k_1,k_2)}$ gives an upper bound for the COP-CP minimization problem (\ref{eq:lp_min}). 
This implication can be further developed when (\ref{eq:lp_max_relx})-(\ref{eq:lp_min_relx}) has no duality gap.
\begin{theorem}\label{thm:alleq}
For the coupled relaxation order $(k_1,k_2)$, let $(X,\b{z},Z,\b{\xi})^{(k_1,k_2)}$ be the maximizer of (\ref{eq:lp_max_relx}) and let $(Y,\b{w},W,\b{\nu})^{(k_1,k_2)}$ be the minimizer of (\ref{eq:lp_min_relx}).
If both (\ref{eq:ft_max}) and (\ref{eq:ft_min}) hold, and there is no duality gap between the SDP duality pairs (\ref{eq:lp_max_relx})-(\ref{eq:lp_min_relx}), i.e., $\mcal{W}_{\max}^{(k_1,k_2)} = \mcal{W}_{\min}^{(k_1,k_2)}$,
then \be\label{eq:all_eq} \mcal{W}_{\max}^{(k_1,k_2)} = \mcal{W}_{\max}^*  =\mcal{Z}^*_{\max\min} =\mcal{Z}^*_{\min\max} = \mcal{W}_{\min}^* = \mcal{W}_{\min}^{(k_1,k_2)}.\ee
Also, $(X^{(k_1,k_2)},Y^{(k_1,k_2)})$ solves the max-min problem (\ref{eqn:conic}) and the min-max problem (\ref{eqn:conic_minmax}). 
\end{theorem}
The duality requirements between SDP pairs are mild and typically satisfied in practice; indeed they are generally assumed in related literature that employs analogous relaxations (e.g., \citep{laraki2012semidefinite}). We therefore focus on conditions that guarantee the FTC itself.

\begin{assumption}\label{as:inIQ}
There exist a maximizer $(X,z,Z)^*$ for (\ref{eq:lp_max}) and a minimizer $(Y,w,W)^*$ for (\ref{eq:lp_min}) such that $p_{Z^*}\coloneqq \psi_m(Z^*) \in \iq[\Delta^m]$, $p_{W^*}\coloneqq \psi_n(W^*) \in \iq[\Delta^n]$, where \(\iq[\Delta^n]:=\bigcup_{k\ge 1}\iq[\Delta^n]_{2k}\) and similarly for
\(\iq[\Delta^m]\).
Moreover, for each one of the following two optimization problems, either its minimum value is strictly positive, or the set of critical points
at which the objective value equals zero is finite. %
\begin{align}
& \min_{\bx\in\re^n} \ p_{W^*}(\bx)\qquad \st \quad  \bx\in\Delta^n, \label{eq:minPsi}\\
& \min_{\by\in\re^m} \ p_{Z^*}(\by)\qquad \st \quad  \by\in\Delta^m. \label{eq:minPhi}
\end{align}
\end{assumption}

We remark that at the optimal solutions $(X,z,Z)^*$ and $(Y,w,W)^*$, the polynomials $p_{W^*}$ and $p_{Z^*}$ are nonnegative on $\Delta^n$ and $\Delta^m$, respectively.
If a polynomial is nonnegative on the simplex set, then it belongs to $\iq[\Delta^n]$ or $\iq[\Delta^m]$ and has finitely many critical points when some {\it generic} conditions hold; see \cite[Theorem~1.2]{nie2014optimality}.
The following result demonstrates that the coupled FTCs hold at all optimal solutions of (\ref{eq:lp_max_relx})-(\ref{eq:lp_min_relx}) when the relaxation orders are big enough, under Assumption~\ref{as:inIQ}.

\begin{theorem}\label{tm:finiteconv}
Suppose that $\mcal{W}^{^*}_{\max} = \mcal{W}^{^*}_{\min}$ and Assumption~\ref{as:inIQ} hold.
The FTCs (\ref{eq:ft_max})-(\ref{eq:ft_min}) hold at all optimal solutions of (\ref{eq:lp_max_relx})-(\ref{eq:lp_min_relx}) for all sufficiently large $k_1$ and $k_2$.
Furthermore, in this case we have $\mathcal W_{\max}^{(k_1,k_2)} = \mathcal W_{\min}^{(k_1,k_2)}$, and all conclusions in Theorem~\ref{thm:alleq} hold.
\end{theorem}

It is also worth emphasizing that once FTC holds for $\b{\xi}$, one can get the decomposition as in (\ref{eq:R2hom}) by the method in \citep{henrion2005detecting}, using Cholesky decomposition. 
The analogous also holds when $\b{\nu}$ satisfies the FTC. In the CP context, this is particularly valuable: because the CP cone is isomorphic to the cone $\mcal{R}[\Delta^n]_2^{\hom}$, such a decomposition for the moment matrix gives a rank-one decomposition of the optimal CP matrices.  
For example, 
in the bi-quadratic game
application of Section 5, such decompositions recover the pure strategies in
the support of an equilibrium mixed strategy, together with their
probabilities.

\section{Zero-Sum Bi-Quadratic Games and Max-Min Bilinear CP programs}
\label{sec:game}
This section considers zero-sum \emph{bi-quadratic
games} with mixed-binary strategy sets, and how to reformulate such games as max--min bilinear CP programs. 
Let $\bs_1\in\re^{l_1}$ and $\bs_2\in \re^{l_2}$  be the strategy vectors of Players 1 and 2, respectively, and let $\mathbb{B}_1\subseteq [{l_1}]$, $\mathbb{B}_2\subseteq [{l_2}]$ be the index sets of binary variables. The pure-strategy feasible sets are $\mcal S_1$ and $\mcal S_2$ defined as follows.
\[\begin{aligned}
\mathcal S_1 : = & \{\bs_1 \in \re^{l_1} : G \bs_1 = \b{g},\ \bs_1\ge 0,\ \bs_{1,i}\in\{0,1\}\ \ \forall i\in \mathbb{B}_1 \},\\
\mathcal S_2 : = & \{\bs_2 \in \re^{l_2} : H \bs_2 = \b{h},\ \bs_2\ge 0,\ \bs_{2,j}\in\{0,1\}\ \forall j\in \mathbb{B}_2 \}.
\end{aligned}\]
\rd{Here, $G\in\re^{d_1\times l_1}$, $H\in\re^{d_2\times l_2}$, $\b g\in\re^{d_1}$, and $\b h\in\re^{d_2}$.}
Player 1 maximizes, while Player 2 minimizes, the following bi-quadratic payoff function:
\begin{equation}\label{eqn:payof}
\begin{array}{rll}
P(\bs_1,\bs_2)\coloneqq \bs_1^\mathsf T C \bs_2+ \operatorname{vec}(\bs_1\bs_1^\mathsf T )^\mathsf T F \operatorname{vec}(\bs_2 \bs_2^\mathsf T).
\end{array}
\end{equation}
In the above, $C\in \re^{l_1\times l_2}$ and $F\in \re^{l_1^2\times l_2^2}$ are matrices that encode the first and the second order interactions, respectively, in the payoff function.

We first describe two cyclic Blotto variants that motivate this class of
bi-quadratic games.
\begin{example}[Binary cyclic Blotto game]
\label{exp:B-cyclic}
Consider $N$ locations arranged in a circle. Throughout the cyclic game, indices are understood modulo $N$; in particular, when $i=1$, the previous location $i-1$ is interpreted as $N$.
In this binary version, the defender and attacker simultaneously
decide whether to occupy each location  $i=1,\ldots,N$, subject to occupation budgets $E_1$ and $E_2$, respectively.
The defender receives a payoff of $1$ at location $i$ if the defender occupies $i$ while the attacker does not, and a payoff of $0$ if the two players' allocations result in a tie.
If the attacker occupies location $i$ and the defender does not, 
the defender may still neutralize the attack using
support from location $i-1$, provided that the defender occupies $i-1$ and
the attacker does not. 
Otherwise, the attacker wins location $i$, resulting in a payoff of $-1$ for the defender.
Let $s_{1,i}\in\{0,1\}$ and $s_{2,i}\in\{0,1\}$ indicate whether the defender and attacker occupy location $i$, respectively.
Then the defender's payoff on location $i$ is:
\be\label{eq:payoff_i}
\resizebox{0.92\linewidth}{!}{$\displaystyle  \mbox{Defender's payoff on location $i$} = 
\left\{ \begin{array}{rl}
1, & \quad \mbox{ if }s_{1,i} = 1,\ s_{2,i} = 0;\\  
-1, & \quad \mbox{ if }s_{1,i-1} = s_{1,i} = 0,\ s_{2,i} = 1;\\
-1, & \quad \mbox{ if }s_{1,i}=0,\ s_{1,i-1}=s_{2,i-1}=s_{2,i}=1;\\
0, & \quad \mbox{ otherwise. }
\end{array}
\right. $}
\ee
In the above, the second loss case occurs when the defender cannot use support from location $i-1$ because that location is also occupied by the attacker.
The attacker's payoff is the negative of the defender's.
This support mechanism confers a structural advantage to the defender. 

Binary cyclic Blotto games can be formulated as bi-quadratic games (\ref{eqn:payof}).
For binary variables $s_{1,i}, s_{2,i}$ in the above,
let $\b{s}_1\coloneqq (s_{1,1}\ddd s_{1,N})$ and $\b{s}_2\coloneqq (s_{2,1}\ddd s_{2,N})$.
Then the feasible sets for the defender and attacker are
\[
\mcal{S}_1\coloneqq \left\{ \b{s}_1 \in \{0,1\}^{N}
\left|\ 
\begin{array}{lll}
s_{1,1} + \cdots + s_{1,N} \le E_1
\end{array}
\right.\right\},
\]\[
\mcal{S}_2\coloneqq \left\{ \b{s}_2 \in \{0,1\}^{N}
\left|\ 
\begin{array}{lll}
s_{2,1} + \cdots + s_{2,N} \le E_2
\end{array}
\right.\right\}.
\]
Moreover, for the pair of feasible strategies $(\b{s}_1,\b{s}_2)$, it follows from (\ref{eq:payoff_i}) that the total payoff for the defender, with the cyclic convention $s_{\ell,0}:=s_{\ell,N}$ for $\ell=1,2$, is
\begin{equation}\label{eqn:payoffblotto01}
\resizebox{0.92\linewidth}{!}{$\displaystyle 
P(\b{s}_1,\b{s}_2) = 
\sum\limits_{i=1}^N \left(s_{1,i}(1-s_{2,i}) -
 (1-s_{1,i})(1-s_{1,i-1})s_{2,i}- (1-s_{1,i})s_{1,i-1}s_{2,i}s_{2,i-1} \right).
$}
\end{equation}
\end{example}

\begin{example}[Cyclic Blotto game]
\label{exp:cyclic}

The binary cyclic Blotto game can be extended to a general resource allocation setting, where the defender and attacker allocate $E_1$ and $E_2$ units of resources, respectively, among $N$ locations arranged in a circle.
We refer to this model as the (general) cyclic Colonel Blotto game
The cyclic support mechanism remains the same: when the defender is outmatched at location $i$, resources from the previous location $i-1$ can be used jointly with the allocation at location $i$ to determine the outcome.
An illustrative example of the local and overall defender payoff calculation is shown in Figure~\ref{fig:cyclic_example}.

\begin{figure}[htbp]
\centering

\begin{minipage}[t]{0.45\textwidth}
\centering
\vphantom{\Large $N=4,E_1=3,E_2=3$}
\textbf{(a) $N=3$, $E_1=2$, $E_2=4$}

\vspace{0.2cm}

\begin{tikzpicture}[
 scale=0.55,
 transform shape,
 >=Latex,
 every node/.style={font=\large},
 state/.style={
    circle,
    draw=black,
    very thick,
    minimum size=14mm,
    inner sep=0pt
 },
 lab/.style={font=\large}
]

\node[state] (1) at (90:3)   {I};
\node[state] (2) at (-30:3)  {II};
\node[state] (3) at (210:3)  {III};

\draw[->, very thick] (1) to[bend left=15] (2);
\draw[->, very thick] (2) to[bend left=15] (3);
\draw[->, very thick] (3) to[bend left=15] (1);

\node[lab, below=3mm of 1] {0};
\node[lab, above=3mm of 1] {2};

\node[lab, above left=3mm of 2] {2};
\node[lab, below right=3mm of 2] {0};

\node[lab, above right=3mm of 3] {2};
\node[lab, below left=3mm of 3] {0};

\end{tikzpicture}

\end{minipage}
\hfill
\begin{minipage}[t]{0.45\textwidth}
\centering
\textbf{(b) $N=4$, $E_1=3$, $E_2=3$}

\vspace{0.2cm}

\begin{tikzpicture}[
 scale=0.55,
 transform shape,
 >=Latex,
 every node/.style={font=\large},
 state/.style={
    circle,
    draw=black,
    very thick,
    minimum size=14mm,
    inner sep=0pt
 },
 lab/.style={font=\large}
]

\node[state] (1) at (90:3)   {I};
\node[state] (2) at (0:3)    {II};
\node[state] (3) at (-90:3)  {III};
\node[state] (4) at (180:3)  {IV};

\draw[->, very thick] (1) to[bend left=15] (2);
\draw[->, very thick] (2) to[bend left=15] (3);
\draw[->, very thick] (3) to[bend left=15] (4);
\draw[->, very thick] (4) to[bend left=15] (1);

\node[lab, below=3mm of 1] {0};
\node[lab, above=3mm of 1] {2};

\node[lab, left=3mm of 2] {1};   %
\node[lab, right=3mm of 2] {0};  %

\node[lab, above=3mm of 3] {1};  %
\node[lab, below=3mm of 3] {1};  %

\node[lab, right=3mm of 4] {1};  %
\node[lab, left=3mm of 4] {0};   %

\end{tikzpicture}

\end{minipage}

\vspace{0.4cm}

\small

\begin{tabular}{c c c c c}
\toprule
\multicolumn{5}{c}{(a): $N=3,\ E_1=2,\ E_2=4$ (Left)}\\
\midrule
Location & 
\begin{tabular}{c}
Defender\\
allocation
\end{tabular} 
& \begin{tabular}{c}
Attacker\\
allocation
\end{tabular}  
& \begin{tabular}{c}
Comparison\\
(Defender vs Attacker)
\end{tabular}
& \begin{tabular}{c}
Defender\\
payoff
\end{tabular}\\
\midrule

I
&2&0&$2>0$ (No support)&1\\

II
&0&2&$0+2=2+0$&0\\

III
&0&2&$0+0<2+2$&-1\\

\midrule
Total
&$E_1=2$
&$E_2=4$
&--
&0\\

\bottomrule
\end{tabular}

\vspace{0.3cm}

\begin{tabular}{c c c c c}
\toprule
\multicolumn{5}{c}{(b): $N=4,\ E_1=3,\ E_2=3$ (Right)}\\
\midrule
Location
& \begin{tabular}{c}
Defender\\
allocation
\end{tabular} 
& \begin{tabular}{c}
Attacker\\
allocation
\end{tabular}  
& \begin{tabular}{c}
Comparison\\
(Defender vs Attacker)
\end{tabular}
& \begin{tabular}{c}
Defender\\
payoff
\end{tabular}\\
\midrule

I
&2&0&$2>0$ (No support)&1\\

II
&0&1&$0+2>1+0$&1\\

III
&1&1&$1=1$ (No support)&0\\

IV
&0&1&$0+1<1+1$&-1\\

\midrule
Total
&$E_1=3$
&$E_2=3$
&--
&1\\

\bottomrule
\end{tabular}

\caption{
The left panel and the right panel correspond to instances with 
$N=3$, $E_1=2$, $E_2=4$ and 
$N=4$, $E_1=3$, $E_2=3$.
In each panel, the numbers inside and outside the ring represent the attacker’s and defender’s resource allocations, respectively.
The directed edges indicate the cyclic support relationship, where resources from the previous location can be used to support the next location.
The accompanying tables summarize the effective resource comparisons and the resulting defender payoffs.
}
\label{fig:cyclic_example}

\end{figure}

Cyclic Blotto games can be formulated as zero-sum bi-quadratic game with binary variables.
For every $k = 0,1\ddd E_1$ and each $i = 1\ddd N$,
let $s_{1,k,i}$ be the indicator variable that takes the value $1$ if the defender deploys $k$ resources at location $i$, and it equals zero otherwise.
Similarly, we define $s_{2,l,j}$ for the attacker's allocation of $l$ resources at location $j$, for each $0\le l\le E_2$ and $j = 1\ddd N$. 
Denote by 
\[  \b{s}_1\coloneqq (s_{1,k,i})_{k = 0,1\ddd E_1,\, i = 1\ddd N},
\quad \b{s}_2\coloneqq (s_{2,l,j})_{l = 0,1\ddd E_2,\, j = 1\ddd N}.\]
For feasible allocations $\b{s}_1$ and $\b{s}_2$, 
the sum of all allocations cannot exceed the total number of resources.
Moreover, for each location $i$, $\sum\limits_{k=0}^{E_1}s_{1,k,i} = \sum\limits_{l=0}^{E_2}s_{2,l,i} = 1$ since allocation strategies at one single location are mutually exclusive. 
Thus, the set of feasible pure strategies for the defender and the attacker are: 
\[
\resizebox{\linewidth}{!}{$\displaystyle 
\mcal{S}_1\coloneqq \left\{ \b{s}_1 \in \{0,1\}^{N(E_1+1)}
\left|\ 
\begin{array}{lll}
\sum\limits_{i=1}^N\sum\limits_{k=0}^{E_1}ks_{1,k,i} \le E_1,\ 
\sum\limits_{k=0}^{E_1}s_{1,k,1} = \cdots = \sum\limits_{k=0}^{E_1}s_{1,k,N} = 1
\end{array}
\right.\right\},
$}
\]
\[
\mcal{S}_2\coloneqq \left\{ \b{s}_2 \in \{0,1\}^{N(E_2+1)}
\left|\ 
\begin{array}{lll}
\sum\limits_{j=1}^N\sum\limits_{l=0}^{E_2}ls_{2,l,j} \le E_2,\ 
\sum\limits_{l=0}^{E_2}s_{2,l,1} = \cdots = \sum\limits_{l=0}^{E_2}s_{2,l,N} = 1
\end{array}
\right.\right\}.
\]
Given feasible allocations $\b{s}_1\in\mcal{S}_1$ and $\b{s}_2\in\mcal{S}_2$, if the defender allocates more resources than the attacker at location $i$,
then the defender wins the location $i$, which can be reflected by
\[ P_i^{direct}(\b{s}_1,\b{s}_2) \coloneqq \sum\limits_{ k > l } s_{1,k,i}s_{2,l,i}. \]
Otherwise, when the defender is outmatched at location $i$, the defender can seek support from location $i-1$, and the payoff in these cases becomes
\[\begin{aligned} P_i^{support}(\b{s}_1,\b{s}_2)  &\coloneqq \sum\limits_{k < l, k + {k'} > l+ l'  } s_{1,k,i}s_{1,k',(i-1)}s_{2,l,i}s_{2,l',(i-1)} \\ &\quad-
\sum\limits_{k < l, k + {k'} < l+ l'  } s_{1,k,i}s_{1,k',(i-1)}s_{2,l,i}s_{2,l',(i-1)}. \end{aligned}\]
Then, it is clear that the defender's total payoff is
\begin{equation*}\label{eqn:payoffblotto} P(\b{s}_1,\b{s}_2)  = \sum_{i=1}^N \big(P_i^{direct}(\b{s}_1,\b{s}_2) + P_i^{support}(\b{s}_1,\b{s}_2)\big).\end{equation*}
The payoff function contains not only bilinear terms (given by $P_i^{direct}$), 
but also bi-quadratic terms resulting from the complementary actions taken at two adjacent locations (i.e., $P_i^{support}$).

\end{example}

{\subsection{Max-Min Bilinear CP Program Reformulations}
\label{sc:maxmin_game_reform}
We formulate the zero-sum bi-quadratic game as a max--min bilinear CP program studied in Sections~\ref{sec:formulation}--\ref{sec:conv_analy}.
For zero-sum bi-quadratic games, {\it mixed-strategy Nash equilibria} are usually considered:
A mixed strategy for Player $i$ is a probability measure $\mu_i$ on $\mathcal S_i$ with finite second moments.
Let $\mathcal M_i$ denote the set of these measures.
The two players randomize independently.
A pair $(\mu_1^*,\mu_2^*)$ is a mixed-strategy Nash equilibrium if the following holds for any $(\mu_1,\mu_2)\in\mathcal M_1\times\mathcal M_2$:
\begin{equation}\label{eq:mix_bigame}
\bE_{\mu_1\otimes\mu_2^*}[P(\bs_1,\bs_2)]\le \bE_{\mu_1^*\otimes\mu_2^*}[P(\bs_1,\bs_2)]\le \bE_{\mu_1^*\otimes\mu_2}[P(\bs_1,\bs_2)].
\end{equation}

To obtain the max--min bilinear CP program reformulation for bi-quadratic games,  we consider the following assumption throughout this section.
\begin{assumption}\label{as:0to1}
Every $u\ge0$ with $Gu=\b g$ satisfies $u_i\le1$ for all $i\in\mathbb B_1$, and every $v\ge0$ with $Hv=\b h$ satisfies $v_j\le1$ for all $j\in\mathbb B_2$.
\end{assumption}
This assumption can be enforced without loss of generality by adding upper bounds on the binary coordinates and nonnegative slack variables to express these bounds as equalities; see \citep{burer2009copositive}.

Consider the CP feasible sets
\begin{equation}\label{eqn:cpconex_qg}
\begin{array}{cc}
\mathcal C_{cp}^1  = \left\{\begin{array}{cc} (\b{p}, P) \in \re^{l_1} \times \mathbb{S}^{l_1} & \left| \begin{array}{cc}
G\b{p} = \b{g},\\
diag(G P {G}^\mathsf T) =  \b{g} \bullet \b{g},\\
\rd{GP=\b g\bp^\tp,}\\
{p}_{i}- P_{i,i} = 0,\ \forall i\in \mathbb{B}_1,\\
\left[\begin{array}{cc} 1 & {\b{p}}^\mathsf T\\ \b{p} & P\end{array}\right] \succcurlyeq_{cp} 0
\end{array}\right.
\end{array}\right\},
\end{array}
\end{equation}
\begin{equation}\label{eqn:cpconey_qg}
\begin{array}{rll}
\mathcal C_{cp}^2 = \left\{\begin{array}{cc} (\b{q}, Q) \in \re^{l_2} \times \mathbb{S}^{l_2} & \left| \begin{array}{cc}
H\b{q} = \b{h},\\
diag(H Q {H}^\mathsf T) =  \b{h} \bullet \b{h},\\
\rd{HQ=\b h\bq^\tp,}\\
{q}_{j}- Q_{j,j} = 0,\ \forall j\in \mathbb{B}_2,\\
\left[\begin{array}{cc} 1 & {\b{q}}^\mathsf T\\ \b{q} & Q \end{array}\right] \succcurlyeq_{cp} 0
 \end{array}\right.
\end{array}\right\}.
\end{array}
\end{equation}
\rd{In the above, the equalities $GP=\b g\bp^\tp$ and $HQ=\b h\bq^\tp$ are redundant since they are implied by the first two equations together with the CP constraint in each $\mathcal C_{cp}^i$.
We include these redundant equalities to improve the numerical stability of the semidefinite relaxations proposed in Section~\ref{sec:formulation}.}
Let $\mcal{F}$ be the linear operator induced
by the matrix $F$ such that
$\lip \mcal{F} (P), Q \rip = \lip  P, \mcal{F}^*(Q) \rip = \operatorname{vec}(P)^\mathsf T F \operatorname{vec}(Q)$ for all $P$ and $Q$.
Define the bilinear objective 
\[ V((\b{p}, P),(\b{q}, Q)) \,:=\, \b{p}^{\tp} C \b{q} + \lip \mcal{F} (P), Q \rip = \b{p}^{\tp} C \b{q} + \lip P, \mcal{F}^{\star}(Q) \rip. \]
Consider the max--min bilinear CP program
\be\label{eq:cpmaximin}
\mathcal V_1^*:=\max_{(\bp,P)\in\mathcal C_{cp}^1}\min_{(\bq,Q)\in\mathcal C_{cp}^2}
V((\bp,P),(\bq,Q)),
\ee
and its associated min--max program
\be\label{eq:cpminimax}
\mathcal V_2^*:=\min_{(\bq,Q)\in\mathcal C_{cp}^2}\max_{(\bp,P)\in\mathcal C_{cp}^1}
V((\bp,P),(\bq,Q)).
\ee
It is clear that (\ref{eq:cpmaximin}) and (\ref{eq:cpminimax}) are max-min and min-max bilinear CP programs, respectively. 

The following proposition uses Burer's CP representation
\citep[Theorem~2.6]{burer2009copositive} to show that the first and
second moments of any mixed-strategy Nash equilibrium for (\ref{eq:mix_bigame}) are optimal
solutions for (\ref{eq:cpmaximin}) and (\ref{eq:cpminimax}).
\begin{proposition}\label{prop:bqg=cpg}
Suppose Assumption~\ref{as:0to1} holds.
If $(\mu_1^*,\mu_2^*)\in\mathcal M_1\times\mathcal M_2$ is a mixed-strategy Nash equilibrium of the bi-quadratic game, then
\[
\begin{aligned}
\bp^*&=\bE_{\mu_1^*}[\bs_1], & P^*&=\bE_{\mu_1^*}[\bs_1\bs_1^\tp],&
\bq^*&=\bE_{\mu_2^*}[\bs_2], & Q^*&=\bE_{\mu_2^*}[\bs_2\bs_2^\tp]
\end{aligned}
\]
are optimal solutions of \eqref{eq:cpmaximin} and \eqref{eq:cpminimax}, respectively, and
\[
\mathcal V_1^*=\mathcal V_2^*=\bE_{\mu_1^*\otimes\mu_2^*}[P(\bs_1,\bs_2)]=V((\bp^*,P^*),(\bq^*,Q^*)).
\]
\end{proposition}

\begin{remark}
Conversely, suppose $\mathcal V_1^*=\mathcal V_2^*$ and let $(\bp^*,P^*)$ and $(\bq^*,Q^*)$ be optimal solutions of \eqref{eq:cpmaximin} and \eqref{eq:cpminimax}, respectively. Then, for all $(\bp,P)\in\mathcal C_{cp}^1$ and $(\bq,Q)\in\mathcal C_{cp}^2$, it holds
\[
V((\bp,P),(\bq^*,Q^*))
\le V((\bp^*,P^*),(\bq^*,Q^*))
\le V((\bp^*,P^*),(\bq,Q)).
\]
Therefore, \citep[Theorem~2.6(ii)]{burer2009copositive} further guarantees
\[
\begin{aligned}
\bp^*&\in\operatorname{conv}\!\Big(\arg\max_{\bs_1\in\mathcal S_1}
V((\bs_1,\bs_1\bs_1^\tp),(\bq^*,Q^*))\Big),\\
\bq^*&\in\operatorname{conv}\!\Big(\arg\min_{\bs_2\in\mathcal S_2}
V((\bp^*,P^*),(\bs_2,\bs_2\bs_2^\tp))\Big).
\end{aligned}
\]
These convex combinations yield mixed strategies with first moments $\bp^*,\bq^*$, but need not match $P^*,Q^*$ or form a mixed-strategy equilibrium. 
Nonetheless, if $\mcal{S}_1$ and $\mcal{S}_2$ are compact, then Nash equilibria of (\ref{eq:mix_bigame}) can be recovered from $\bp^*,\bq^*,P^*,Q^*$.
This is shown in Theorem~\ref{thm:gameNE}.
\end{remark}

Computing a mixed-strategy Nash equilibrium of the bi-quadratic game
directly may be difficult. We therefore consider
the paired max--min and min--max problems
\eqref{eq:cpmaximin}--\eqref{eq:cpminimax}. For (\ref{eq:cpmaximin}), applying the COP--CP
reformulation developed in Section~\ref{sec:conic} gives
\be\label{eq:cpmaxmax}
\left\{
\begin{array}{rcl}
\mcal U_1^{cop-cp} \,:= \, & \max  & \quad \rho + {\b{h}}^\mathsf T \b{r} + (\b{h}\bullet \b{h})^\mathsf T \b{v} \\
& \st & \quad {(\b{p},P)\in \mcal{C}_{cp}^1} \\
&     & \quad (\rho, \b{r},\b{v},\b{\vartheta}\rd{,\Xi})\in \mcal{C}_{cop}^2(\b{p},P).
\end{array}
\right.\ee
\rd{Applying the same construction to (\ref{eq:cpminimax}) yields}
\be\label{eq:cpminmin} 
\left\{
\begin{array}{rcl}
\mcal U_2^{cop-cp} \,:= \, & \min  & \quad \tau + {\b{g}}^\mathsf T \b{t} + (\b{g}\bullet \b{g})^\mathsf T \b{u}\\
& \st & \quad {(\b{q},Q)\in \mcal{C}_{cp}^2} \\
&     & \quad (\tau, \b{t},\b{u},\b{\kappa}\rd{,\Gamma})\in \mcal{C}_{cop}^1(\b{q},Q).
\end{array}
\right.\ee
In the above, we define
\rd{
\[\arxivfitmath{ \begin{gathered}
\mcal C_{cop}^1(\bq,Q):=\left\{
(\tau,\b t,\b u,\b\kappa,\Gamma) \in \re\times\re^{d_1}\times\re^{d_1}\times\re^{l_1}\rd{\times\re^{d_1\times l_1}}\;\middle|\;
\begin{gathered}
\Theta^1(\tau,\b t,\b u,\b\kappa,\Gamma)\succcurlyeq_{co}0,\\
\kappa_i=0\quad(i\notin\mathbb B_1)
\end{gathered}
\right\}, \\
\mcal C_{cop}^2(\bp,P):=\left\{
(\rho,\b r,\b v,\b\vartheta,\Xi)\in\re\times\re^{d_2}\times\re^{d_2}\times\re^{l_2}\rd{\times\re^{d_2\times l_2}}\;\middle|\;
\begin{gathered}
\Theta^2(\rho,\b r,\b v,\b\vartheta,\Xi)\succcurlyeq_{co}0,\\
\vartheta_j=0\quad(j\notin\mathbb B_2)
\end{gathered}
\right\},
\end{gathered}
}\]
where for fixed $(\bq,Q)$ and $(\bp,P)$, $\Theta^1$ and $\Theta^2$ are symmetric $2$-by-$2$ block matrices such that
\[\arxivfitmath{
\begin{gathered}
\Theta^1_{11}=\tau,\quad
\Theta^1_{12}=(\Theta^1_{21})^\tp
=\tfrac12(G^\tp\b t+\b\kappa-C\bq-\Gamma^\tp\b g)^\tp,\\
\Theta^1_{22}=G^\tp\Lambda(\b u)G-\Lambda(\b\kappa)-\mcal F^\star(Q)
+\tfrac12(G^\tp\Gamma+\Gamma^\tp G);
\end{gathered}
}\]
\[
\resizebox{\linewidth}{!}{$\displaystyle 
\begin{gathered}
\Theta^2_{11}=-\rho,\quad
\Theta^2_{12}=(\Theta^2_{21})^\tp
=\tfrac12(C^\tp\bp-H^\tp\b r-\b\vartheta+\Xi^\tp\b h)^\tp,\\
\Theta^2_{22}=\mcal F(P)-H^\tp\Lambda(\b v)H+\Lambda(\b\vartheta)
-\tfrac12(H^\tp\Xi+\Xi^\tp H).
\end{gathered}
$}
\]
}%
\begin{proposition}\label{prop:gamecpco}
The COP--CP programs (\ref{eq:cpmaxmax}) and (\ref{eq:cpminmin}) are conic duals. Moreover, 
    \begin{enumerate}
    \item If ${H}^\mathsf T H\succ_{co}0$, then $\mcal{V}_1^*=\mcal U_1^{cop-cp}$.
    \item If ${G}^\mathsf T G\succ_{co}0$, then $\mcal{V}_2^*=\mcal U_2^{cop-cp}$.
    \end{enumerate}
\end{proposition}
Interestingly, under Assumption~\ref{as:0to1}, if $\mcal S_1$ and $\mcal S_2$ are nonempty and compact, then
$G^\tp G\succ_{co}0$ and $H^\tp H\succ_{co}0$, which is shown in Theorem~\ref{thm:gameNE}.
In this case, Proposition~\ref{prop:gamecpco} ensures the equivalence between the max-min bilinear CP program (\ref{eq:cpmaximin}) and the linear COP-CP program (\ref{eq:cpmaxmax}) (analogously, (\ref{eq:cpminimax}) and (\ref{eq:cpminmin})).

Compactness also strengthens the moment correspondence.
The following theorem relates
optimal solutions of the CP and COP--CP programs to mixed-strategy
Nash equilibria of the bi-quadratic game.
We say that a pair $(\bp^*,P^*)$ admits a maximizer of \eqref{eq:cpmaxmax}
if there exist auxiliary variables $(\rho^*,\b r^*,\b v^*,\b\vartheta^*\rd{,\Xi^*})$
such that they give a maximizer of \eqref{eq:cpmaxmax}; an analogous definition applies
to a pair $(\bq^*,Q^*)$ admitting a minimizer of \eqref{eq:cpminmin}.

\begin{theorem}\label{thm:gameNE}
Suppose $\mcal S_1$ and $\mcal S_2$ are nonempty and compact. Then:
\begin{enumerate}
\item[i.]
$\mcal V_1^*
=
\mcal U_1^{cop-cp}
=
\mcal U_2^{cop-cp}
=
\mcal V_2^*$.
\item[ii.]
A pair
$(\mu_1^*,\mu_2^*)\in\mcal M_1\times\mcal M_2$
is a Nash equilibrium of game \eqref{eq:mix_bigame} if and
only if there exist optimal solutions $(\bp^*,P^*)$ and $(\bq^*,Q^*)$ of \eqref{eq:cpmaximin} and \eqref{eq:cpminimax}, respectively, such that
\[
\begin{aligned}
\bp^*&=\int_{\mcal S_1}\bs_1\,d\mu_1^*(\bs_1),
& P^*&=\int_{\mcal S_1}\bs_1\bs_1^\tp\,d\mu_1^*(\bs_1),\\
\bq^*&=\int_{\mcal S_2}\bs_2\,d\mu_2^*(\bs_2),
& Q^*&=\int_{\mcal S_2}\bs_2\bs_2^\tp\,d\mu_2^*(\bs_2).
\end{aligned}
\]
\item[iii.]
If $(\bp^*,P^*)$ admits a maximizer of \eqref{eq:cpmaxmax} and
$(\bq^*,Q^*)$ admits a minimizer of \eqref{eq:cpminmin}, then they are
optimal solutions of \eqref{eq:cpmaximin} and \eqref{eq:cpminimax},
respectively, and are the first and second moments of a
mixed-strategy Nash equilibrium of \eqref{eq:mix_bigame}, as in \textup{(ii)}.
\end{enumerate}
\end{theorem}

Theorem~\ref{thm:gameNE} suggests that one can apply the COP--CP
reformulation and semidefinite relaxation method developed in
Section~\ref{sec:formulation} to compute mixed-strategy Nash equilibria
of bi-quadratic games. Section~\ref{sec:numerical} presents numerical
experiments illustrating this application.}

\subsection{Relation to zero-sum polynomial games}
\label{sec:relation_lasserre}
In \citep{laraki2012semidefinite}, Laraki and Lasserre studied zero-sum polynomial games.
Specifically, let $K_1\subseteq \re^{n_1}$ and $K_2\subseteq \re^{n_2}$ be two \rd{compact} basic semialgebraic sets, and let $f$ be a polynomial in $(\b{s}_1,\b{s}_2)\in K_1\times K_2$.
The corresponding zero-sum polynomial game is formulated as the following max-min problem:
\be\label{eq:zero_poly}
\max_{\mu_1\in \mathcal{M}(K_1)}
\min_{\mu_2\in \mathcal{M}(K_2)}
\int\int f(\b{s}_1,\b{s}_2)d\mu_2d\mu_1 .
\ee
In the above, $\mathcal{M}(K_1)$ and $\mathcal{M}(K_2)$ are sets of probability measures supported on $K_1$ and $K_2$, respectively.
Furthermore, a coupled semidefinite relaxation method \textcolor{black}{based on moment and SOS relaxations} for solving (\ref{eq:zero_poly}) was studied in \citep{laraki2012semidefinite}.
The bi-quadratic game considered here is a special case of a polynomial game, but its CP reformulation leads to a different semidefinite hierarchy.
The following comparison concerns the general max--min bilinear CP framework developed in Sections~\ref{sec:formulation}--\ref{sec:conv_analy}.

\begin{itemize}
\item For (\ref{eq:zero_poly}), fixing either player's mixed strategy yields a best-response problem equivalent to a polynomial optimization problem (POP).
On the other hand, for fixed $X$ or $Y$, the max-min bilinear CP program becomes a degree-2 homogeneous {\it generalized moment problem} (GMP); see \citep{lasserre2024moment,nie2015linear,huang2024finite}.
Note that POPs are special cases of GMPs: measure formulations of POPs impose only the unit-mass constraint, whereas general GMPs allow additional linear moment constraints.
Thus, the bilinear CP formulation leads to a coupling of GMPs, extending the POP-based structure underlying polynomial games.

\item 
The proof of our Theorem~\ref{tm:asym_conv} for the asymptotic convergence is inspired by its analogous result in \cite[Theorem~5.1]{laraki2012semidefinite} for polynomial games.
However, the general moment constraints introduce two additional challenges.
First, the relaxed moment sequences are not normalized.
Thus, we use Assumption~\ref{as:co} to establish a uniform bound for their zeroth moments, so that the limiting measure can be obtained. 
Second, convergence of the CP variables alone does not establish feasibility on the COP side, and the associated dual multipliers need not converge.
We address this by perturbing the multipliers along a strictly copositive direction, restoring feasibility at the limiting CP matrix while controlling the objective loss. 

\item We further establish sufficient conditions for finite convergence of the proposed semidefinite hierarchy for max-min bilinear CP programs.
In \cite{laraki2012semidefinite}, FTCs are also used to certify the tightness of semidefinite relaxations after its optimal solutions are obtained.
Our analysis goes beyond this a posteriori criterion by establishing sufficient conditions under which the coupled FTCs necessarily hold at every optimal solution of all sufficiently high-order relaxations. 
This guarantees eventual tightness of the hierarchy, rather than merely providing a test for recognizing tightness at a given relaxation order.
\end{itemize}

\section{Numerical Experiments} \label{sec:numerical}

In this section, we apply the semidefinite relaxation to representative COP–CP instances to demonstrate the effectiveness of the approach. We begin with a toy example that illustrates implementation details, and then proceed to a bi-quadratic game application to provide deeper numerical evidence and insights. \rd{The numerical experiments are carried out on a Windows workstation with an AMD Ryzen 7 5800 CPU and 32 GB of RAM, using MATLAB R2023b and MOSEK 9.3. Computed numerical values are reported to four decimal places.}

\begin{example}
\label{ex:toy}
Consider the max-min bilinear CP program with $n=m=3$ given by 
{\small
\[ \mcal{Q}(X) = \begin{bmatrix}
X_{11}-X_{21}-4X_{31}+X_{22} & X_{21}+\frac{1}{2}X_{31}+X_{33} & \frac{1}{2}(X_{11}-X_{21}-X_{32}) \\
X_{21}+\frac{1}{2}X_{31}+X_{33} & -X_{11}-X_{22}+2X_{32} & -\frac{1}{2}(X_{31}+X_{22}+X_{33}) \\
\frac{1}{2}(X_{11}-X_{21}-X_{32}) & -\frac{1}{2}(X_{31}+X_{22}+X_{33}) & X_{31}+X_{32}+2X_{33}
\end{bmatrix}. \]
\[\begin{gathered}A_1 = I_3,\quad A_2 = 
\begin{bmatrix}
1\, & 1\, & 0\, \\
1\, & 2\, & 1\, \\
0\, & 1\, & 1\,
\end{bmatrix}, \quad 
A_3 = 
\left[
\begin{array}{rrr}
1 & 2 & -1 \\
2 & -1 & -1 \\
-1 & -1 & 1
\end{array}
\right], \\ 
A_4 = 
\left[
\begin{array}{rrr}
1 & 2 & -3 \\
2 & -3 & 2 \\
-3 & 2 & 1
\end{array}
\right], \quad 
\b{b} = 
\left[
\begin{array}{rrr}
1 \\ 2 \\ 0 \\ 1
\end{array}
\right],
\end{gathered}\]
\[B_1 = I_3,\quad 
B_2 = 
\left[
\begin{array}{rrr}
0 & 2 & 0 \\
2 & 0 & 2 \\
0 & 2 & 0
\end{array}
\right], \quad 
B_3 = 
\left[
\begin{array}{rrr}
0 & 0 & -1 \\
0 & -1 & 0 \\
-1 & 0 & 0
\end{array}
\right], \quad 
\b{c} = 
\left[
\begin{array}{rrr}
3 \\ 1 \\ -2
\end{array}
\right].
\]}%
We solve the semidefinite relaxations (\ref{eq:lp_max_relx})--(\ref{eq:lp_min_relx}) with $k_1,k_2 = 1\ddd 3$. 
Numerical results are reported Table~\ref{tab:toy_bound}.
\begin{table}[h]
\centering
\small
\caption{Numerical results for solving Example~\ref{ex:toy}. The last column reports whether the FTC holds for $\xi$ and $\nu$, respectively. }
\label{tab:toy_bound}
\begin{tabular}{|c|r|r|c|c|}
\hline
$(k_1,k_2)$ & $\mcal{W}_{\max}^{(k_1,k_2)}$ & $\mcal{W}_{\min}^{(k_1,k_2)}$ & Time (sec) & FT $(\xi,\nu)$\\
\hline
$(1,1)$ & $-0.2912$ & $-0.2912$ & 0.0100 & (N,N)\\
$(2,2)$ & $0.4995$ & $0.4995$ & 0.0100 & (N,N)\\
$(3,2)$ & $0.4994$ & $0.4994$ & 0.0100 & (Y,N)\\
$(2,3)$ & $0.4995$ & $0.4995$ & 0.0500 & (N,Y)\\
$(3,3)$ & $0.4994$ & $0.4994$ & 0.0400 & (Y,Y)\\
\hline
\end{tabular}
\end{table}
Particularly, when $(k_1,k_2) = (3,3)$, the FTCs (\ref{eq:ft_max})--(\ref{eq:ft_min}) hold with
$\rank M_1[\xi]=\rank M_2[\xi]=\rank M_1[\nu]=\rank M_2[\nu]=2,$
and the duality gap between (\ref{eq:lp_max_relx})--(\ref{eq:lp_min_relx}) is less than $2\times10^{-9}$.
By Theorem~\ref{thm:alleq}, 
we get that 
\[ \mcal{W}_{\max}^{(k_1,k_2)} = \mcal{W}_{\max}^*  =\mcal{Z}^*_{\max\min} =\mcal{Z}^*_{\min\max} = \mcal{W}_{\min}^* = \mcal{W}_{\min}^{(k_1,k_2)} = 0.4994. \]
The optimal solution for this max-min bilinear CP program is 
{\small \[\begin{aligned}
 X^* &= 
\left[
\begin{array}{rrr}
    0.0144  &  0.0556  &  0.0000 \\
    0.0556  &  0.3333  &  0.2778 \\
    0.0000  &  0.2778  &  0.6523 \\
\end{array}
\right] \\ &\approx 0.3405 \left[
\begin{array}{rrr}
    0.2053 \\
    0.7947 \\
    0.0000 \\
\end{array}
\right]\left[
\begin{array}{rrr}
    0.2053 \\
    0.7947 \\
    0.0000 \\
\end{array}
\right]^T
+ 
1.3262 \left[
\begin{array}{rrr}
    0.0000 \\
    0.2987 \\
    0.7013 \\
\end{array}
\right]\left[
\begin{array}{rrr}
    0.0000 \\
    0.2987 \\
    0.7013 \\
\end{array}
\right]^T,\end{aligned}\]
\[\begin{aligned}
 Y^* &= 
\left[
\begin{array}{rrr}
    1.0559  &   0.0000  &  0.0453\\
    0.0000  &   1.9094  &  0.2500\\
    0.0453  &   0.2500  &  0.0347
\end{array}
\right] \\ &\approx 2.4421 \left[
\begin{array}{rrr}
    0.0000 \\
    0.8842 \\
    0.1158 \\
\end{array}
\right]\left[
\begin{array}{rrr}
    0.0000 \\
    0.8842 \\
    0.1158 \\
\end{array}
\right]^T
+ 
1.1484 \left[
\begin{array}{rrr}
    0.9589 \\
    0.0000 \\
    0.0411 \\
\end{array}
\right]\left[
\begin{array}{rrr}
    0.9589 \\
    0.0000 \\
    0.0411 \\
\end{array}
\right]^T.\end{aligned}\]}%
Interestingly, we also remark that when $(k_1,k_2) = (3,2)$, the FTC does not hold for $\nu^{(k_1,k_2)}$ but the tightness (more specifically, the CPness of $Y^{(k_1,k_2)}$) can be certified by solving the truncated moment problem.
\end{example}

\begin{example}\label{exp:toy_game}
Consider the bi-quadratic game (\ref{eq:mix_bigame}) with
$l_1=l_2=4$, $\mathbb{B}_1 = \{1,2\}$, $\mathbb{B}_2 = \{1\}$, and
{\small\[\begin{gathered} C =   \lvt\begin{array}{rrrr}
-2 & -1 &  1 &  3\\
 3 & -2 &  1 & -5\\
 2 & -5 & -2 &  3\\
 3 &  2 & -3 &  0
\end{array}\rvt, \quad                              
G = \lvt\begin{array}{rrrr}
1 & 1 & 0 & 0 \\
-1 & -1 & 2 & 1 \\
\end{array}\rvt,\\ 
H = \lvt\begin{array}{rrrr}
1 & 1 & 0 & 0 \\
0 & -2 & 1 & 1 \\
\end{array}\rvt,\quad
g = \lvt\begin{array}{rrrr}
1 \\2
\end{array}\rvt,\quad
h = \lvt\begin{array}{rrrr}
1 \\3
\end{array}\rvt, \end{gathered}\]}%
{\small \[ 
F =  diag(
\begin{array}{rrrrrrrrrrrrrrrr}
-2 & 5 & 3 & -4 & -2 & 6 & -2 & 1 &
-5 & 2 & 5 & -4 & -3 & 1 & -5 & 5
\end{array}).\]}%
\rd{We consider its associated max-min and min-max bilinear CP programs (\ref{eq:cpmaximin}) and (\ref{eq:cpminimax}).
For $(k_1,k_2)=(3,2)$, we solve (\ref{eq:lp_max_relx})--(\ref{eq:lp_min_relx}), and find the FTCs (\ref{eq:ft_max})--(\ref{eq:ft_min}) hold with
$$\rank M_1[\xi]=\rank M_2[\xi]=2,\quad\rank M_0[\nu]=\rank M_1[\nu]=1.$$
It took around $0.17$ second.
By Theorem~\ref{thm:gameNE}, we get that 
\[\mcal V_1^* = \mcal V_2^* = 43.2793.\]
The resulting optimal moments $(p^*,P^*,q^*,Q^*)$ for (\ref{eq:cpmaximin}) and (\ref{eq:cpminimax}) are}
{\small \[ p^* = \lvt\begin{array}{r}
0.0000 \\
1.0000 \\
\rd{1.0572} \\
\rd{0.8856}
\end{array}\rvt,\quad
P^* =   \lvt\begin{array}{rrrr}
0.0000 & 0.0000 & 0.0000  & 0.0000\\
0.0000 & 1.0000 & \rd{1.0572} & \rd{0.8856}\\
0.0000 & \rd{1.0572} & \rd{1.5858} & 0.0000\\
0.0000 & \rd{0.8856} & 0.0000 & \rd{2.6569}
\end{array}\rvt,\] 
\[q^* = \lvt\begin{array}{rrrr}
1.0000 \\
0.0000 \\
\rd{1.9245} \\
1.0755
\end{array}\rvt,\quad
Q^* =   \lvt\begin{array}{rrrr}
1.0000 & 0.0000 & 1.9245 & 1.0755\\
0.0000 & 0.0000 & 0.0000 & 0.0000\\
1.9245 & 0.0000 & \rd{3.7036} & 2.0698\\
1.0755 & 0.0000 & 2.0698 & \rd{1.1567}
\end{array}\rvt.
\]}%
\noindent\rd{This gives the mixed-strategy Nash equilibrium $(\mu^*_1,\mu^*_2)$ of (\ref{eq:mix_bigame}), which is shown in Table~\ref{tab:toy_game_strategies}.}
\begin{table}[htbp]
\centering
\rd{
\small
\caption{Equilibrium mixed strategies for Example~\ref{exp:toy_game}.}
\label{tab:toy_game_strategies}
\begin{tabular}{|c|c|c|}
\hline
 & Pure strategy & Probability \\
\hline
\multirow{2}{*}{Player 1 ($\mu^*_1$)} & $(0,1,1.5000,0)$ & $0.7048$ \\ \cline{2-3}
 & $(0,1,0,3)$ & $0.2952$ \\ \hline
Player 2 ($\mu^*_2$)& $(1,0,1.9245,1.0755)$ & $1$ \\
\hline
\end{tabular}
}
\end{table}

\end{example}

\subsection*{Cyclic Blotto games} The following two examples illustrate our method on cyclic Blotto games. For brevity, we report only the recovered equilibrium strategies and their corresponding payoffs, omitting the intermediate moment matrices and CP decompositions.
\begin{example}\label{ex:bi-blotto-num}
Consider the binary cyclic Blotto game in Example~\ref{exp:B-cyclic} with $N=4$ and $5$ and various occupation budgets.
\rd{We solve the semidefinite relaxations (\ref{eq:lp_max_relx})--(\ref{eq:lp_min_relx}) with $(k_1,k_2)=(2,2)$. 
Nash equilibria are recovered for all cases, and the resulting equilibrium strategies, payoffs, and solver times are reported in Table~\ref{tab:binary_blotto_updated}.}
Particularly, Figure~\ref{fig:numgame01} visualizes the pure strategies in the support of the equilibrium mixed strategies for the cases $E_1=E_2=2$ and $N=4$ or $5$.

\begin{table}[htbp]
\footnotesize
  \centering
  \caption{Computation results for binary cyclic Blotto game in Example~\ref{ex:bi-blotto-num}.
  In the table, columns $p_x$ and $p_y$ denote the probabilities assigned to each listed pure strategy of the defender and attacker in the equilibrium, respectively.
  For example, when $N=4$ and $E_1=E_2=3$, both players independently randomize over the four sets of three occupied locations, assigning probability $1/4$ to each set.
  The time consumption is reported in seconds.}
\label{tab:binary_blotto_updated}
  \setlength{\tabcolsep}{2pt}
  \begin{tabular}{|c|c|c|c|c|c|c|c|}
    \hline
    $N$ & $(E_1,E_2)$ & $p_x$ & $\begin{array}{c}\mbox{Defender}\\ \mbox{pure strategies}
    \end{array}$ & $p_y$ & $\begin{array}{c}\mbox{Attacker}\\ \mbox{pure strategies}
    \end{array}$ & $\begin{array}{c}\mbox{Defender}\\ \mbox{Payoff}
    \end{array}$ & $\begin{array}{c}\mbox{Time (s)}\end{array}$ \\ 
    \hline
    \multirow{4}{*}{4}
    & $(2,2)$ & $1/2$ & $(1,3), (2,4)$ & $1/4$ & $\begin{array}{c}(1,2), (1,4)\\(2,3), (3,4)\end{array}$ & \rd{$1/2$} & \rd{0.96} \\ \cline{2-8}
    & $(2,3)$ & $1/2$ & $(1,3), (2,4)$ & $1/4$ & $\begin{array}{c}(1,2,3), (1,2,4)\\(1,3,4), (2,3,4)\end{array}$ & \rd{$-1/2$} & \rd{0.89} \\ \cline{2-8}
    & $(3,2)$ & $1/4$ & $\begin{array}{c}(1,2,3), (1,2,4)\\(1,3,4), (2,3,4)\end{array}$ & $1/4$ & $\begin{array}{c}(1,2), (1,4)\\(2,3), (3,4)\end{array}$ & \rd{$5/4$} & \rd{0.91} \\ \cline{2-8}
    & $(3,3)$ & $1/4$ & $\begin{array}{c}(1,2,3), (1,2,4)\\(1,3,4), (2,3,4)\end{array}$ & $1/4$ & $\begin{array}{c}(1,2,3), (1,2,4)\\(1,3,4), (2,3,4)\end{array}$ & \rd{$1/4$} & \rd{0.90} \\ 
    \hline\hline
    \multirow{9}{*}{5}
    & $(2,2)$ & $1/5$ & $\begin{array}{c}(1,3), (1,4)\\(2,4), (2,5)\\(3,5)\end{array}$ & $1/5$ & $\begin{array}{c}(1,2), (1,5)\\(2,3), (3,4)\\(4,5)\end{array}$ & \rd{$2/5$} & \rd{4.28} \\ \cline{2-8}
    & $(2,3)$ & $1/5$ & $\begin{array}{c}(1,3), (1,4)\\(2,4), (2,5)\\(3,5)\end{array}$ & $1/5$ & $\begin{array}{c}(1,2,3), (1,2,5)\\(1,4,5), (2,3,4)\\(3,4,5)\end{array}$ & \rd{$-3/5$} & \rd{4.85} \\ \cline{2-8}
    & $(2,4)$ & $1/5$ & $\begin{array}{c}(1,3), (1,4)\\(2,4), (2,5)\\(3,5)\end{array}$ & $1/5$ & $\begin{array}{c}(1,2,3,4), (1,2,3,5)\\(1,2,4,5), (1,3,4,5)\\(2,3,4,5)\end{array}$ & \rd{$-8/5$} & \rd{3.56} \\ \cline{2-8}
    & $(3,2)$ & $1/5$ & $\begin{array}{c}(1,2,4), (1,3,4)\\(1,3,5), (2,3,5)\\(2,4,5)\end{array}$ & $1/5$ & $\begin{array}{c}(1,2), (1,5)\\(2,3), (3,4)\\(4,5)\end{array}$ & \rd{$7/5$} & \rd{4.13} \\ \cline{2-8}
    & $(3,3)$ & $1/5$ & $\begin{array}{c}(1,2,4), (1,3,4)\\(1,3,5), (2,3,5)\\(2,4,5)\end{array}$ & $1/5$ & $\begin{array}{c}(1,2,3), (1,2,5)\\(1,4,5), (2,3,4)\\(3,4,5)\end{array}$ & \rd{$2/5$} & \rd{3.06} \\ \cline{2-8}
    & $(3,4)$ & $1/5$ & $\begin{array}{c}(1,2,4), (1,3,4)\\(1,3,5), (2,3,5)\\(2,4,5)\end{array}$ & $1/5$ & $\begin{array}{c}(1,2,3,4), (1,2,3,5)\\(1,2,4,5), (1,3,4,5)\\(2,3,4,5)\end{array}$ & \rd{$-3/5$} & \rd{4.25} \\ \cline{2-8}
    & $(4,2)$ & $1/5$ & $\begin{array}{c}(1,2,3,4), (1,2,3,5)\\(1,2,4,5), (1,3,4,5)\\(2,3,4,5)\end{array}$ & $1/5$ & $\begin{array}{c}(1,2), (1,5)\\(2,3), (3,4)\\(4,5)\end{array}$ & \rd{$11/5$} & \rd{3.67} \\ \cline{2-8}
    & $(4,3)$ & $1/5$ & $\begin{array}{c}(1,2,3,4), (1,2,3,5)\\(1,2,4,5), (1,3,4,5)\\(2,3,4,5)\end{array}$ & $1/5$ & $\begin{array}{c}(1,2,3), (1,2,5)\\(1,4,5), (2,3,4)\\(3,4,5)\end{array}$ & \rd{$6/5$} & \rd{3.95} \\ \cline{2-8}
    & $(4,4)$ & $1/5$ & $\begin{array}{c}(1,2,3,4), (1,2,3,5)\\(1,2,4,5), (1,3,4,5)\\(2,3,4,5)\end{array}$ & $1/5$ & $\begin{array}{c}(1,2,3,4), (1,2,3,5)\\(1,2,4,5), (1,3,4,5)\\(2,3,4,5)\end{array}$ & \rd{$1/5$} & \rd{4.50} \\ 
    \hline
  \end{tabular}
\end{table}

\begin{figure}[htbp]
\centering

\begin{subfigure}[b]{0.45\textwidth}
\centering
\begin{tikzpicture}[
scale=2, 
every node/.style={circle, draw, inner sep=1pt, minimum size=8pt}, 
font=\small
]

\node (1) at (90:1)  {I};
\node (2) at (180:1) {II};
\node (3) at (270:1) {III};
\node (4) at (360:1) {IV};

\draw[black, thick] (1) -- (3);
\draw[black, thick] (2) -- (4);

\draw[red, thick, bend left=25] (1) to (2);
\draw[red, thick, bend left=25] (2) to (3);
\draw[red, thick, bend left=25] (3) to (4);
\draw[red, thick, bend left=25] (4) to (1);

\end{tikzpicture}
\caption{$(N=4, E_1=2, E_2=2)$}
\end{subfigure}
\hfill
\begin{subfigure}[b]{0.45\textwidth}
\centering
\begin{tikzpicture}[
scale=2, 
every node/.style={circle, draw, inner sep=1pt, minimum size=8pt}, 
font=\small
]

\node (1) at (90:1)  {I};
\node (2) at (162:1) {II};
\node (3) at (234:1) {III};
\node (4) at (306:1) {IV};
\node (5) at (18:1)  {V};

\draw[black, thick] (1) -- (3);
\draw[black, thick] (2) -- (4);
\draw[black, thick] (3) -- (5);
\draw[black, thick] (4) -- (1);
\draw[black, thick] (5) -- (2);

\draw[red, thick, bend left=25] (1) to (2);
\draw[red, thick, bend left=25] (2) to (3);
\draw[red, thick, bend left=25] (3) to (4);
\draw[red, thick, bend left=25] (4) to (5);
\draw[red, thick, bend left=25] (5) to (1);

\end{tikzpicture}
\caption{$(N=5, E_1=2, E_2=2)$}
\end{subfigure}

\caption{
Visualization of representative equilibrium strategy supports for binary cyclic Blotto games in Example~\ref{ex:bi-blotto-num}.
Panel (a) and Panel (b) correspond to the cases $(N,E_1,E_2)=(4,2,2)$ and $(5,2,2)$, respectively.
Each node represents a location arranged in a circle.
Black straight lines and red curved lines represent pure strategies in the support of the defender's and attacker's equilibrium mixed strategies, respectively.
For example, in Panel (a), the black edge connecting locations 2 and 4 represents the defender's pure strategy $(2,4)$, while the red edge connecting locations 1 and 2 represents the attacker's pure strategy $(1,2)$.
All displayed pure strategies are played with equal probability in the corresponding mixed equilibrium strategy.
}
\label{fig:numgame01}
\end{figure}

\end{example}

\begin{example}
Consider the cyclic Blotto game in Example~\ref{exp:cyclic} with $N = 3$ and $N = 4$ under various resource configurations. 
We solve the semidefinite relaxations (\ref{eq:lp_max_relx})--(\ref{eq:lp_min_relx}) with $(k_1,k_2)=(2,2)$.
The recovered Nash equilibria, payoffs, and solver times for $N = 3$ are reported in Table~\ref{tab:cyclic_blotto_three}, while the detailed results for the case $E_1 = E_2 = 3, N = 4$ are shown in Table~\ref{tab:cyclic_blotto_four}. 

\begin{table}[h]
\small
  \centering
  \caption{Computation results for cyclic Blotto game of three battlefields ($N=3$) using semidefinite relaxation.
  In the table, columns $p_x$ and $p_y$ denote the probabilities assigned to each listed pure strategy of the defender and attacker in the equilibrium, respectively.
  For example, when $E_1=E_2=2$, both players independently randomize over the six allocations of two units among the three locations, assigning probability $1/6$ to each allocation.
  The time consumption is reported in seconds.}
\label{tab:cyclic_blotto_three}
  \renewcommand{\arraystretch}{0.5}
  \begin{tabular}{|c|c|c|c|c||c|c|c|c|c|c|}
    \hline
    \multicolumn{5}{|c||}{\rule[-3pt]{0pt}{12pt}Defender Strategies} & \multicolumn{5}{c|}{Attacker Strategies} & \multirow{2}{*}{\shortstack{Time (s)}} \\ \cline{1-10}
    \rule[-3pt]{0pt}{12pt}$E_1$ & $p_x$ & Loc 1 & Loc 2 & Loc 3 & $E_2$ & $p_y$ & Loc 1 & Loc 2 & Loc 3 &  \\ \hline
    \multirow{12}{*}{2} & \multirow{2}{*}{1/6} & \multirow{2}{*}{1} & \multirow{2}{*}{1} & \multirow{2}{*}{0} & \multirow{12}{*}{2} & \multirow{2}{*}{1/6} & \multirow{2}{*}{1} & \multirow{2}{*}{1} & \multirow{2}{*}{0} & \multirow{12}{*}{\rd{3.71}} \\ 
     &  &  &  &  &  &  &  &  & &  \\ \cline{2-5}\cline{7-10}
     & \multirow{2}{*}{1/6} & \multirow{2}{*}{1} & \multirow{2}{*}{0} & \multirow{2}{*}{1} &  & \multirow{2}{*}{1/6} & \multirow{2}{*}{1} & \multirow{2}{*}{0} & \multirow{2}{*}{1} &  \\ 
     &  &  &  &  &  &  &  &  & &  \\ \cline{2-5}\cline{7-10}
     & \multirow{2}{*}{1/6} & \multirow{2}{*}{0} & \multirow{2}{*}{1} & \multirow{2}{*}{1} &  & \multirow{2}{*}{1/6} & \multirow{2}{*}{0} & \multirow{2}{*}{1} & \multirow{2}{*}{1} &  \\ 
     &  &  &  &  &  &  &  &  & &  \\ \cline{2-5}\cline{7-10}
     & \multirow{2}{*}{1/6} & \multirow{2}{*}{2} & \multirow{2}{*}{0} & \multirow{2}{*}{0} &  & \multirow{2}{*}{1/6} & \multirow{2}{*}{2} & \multirow{2}{*}{0} & \multirow{2}{*}{0} &  \\ 
     &  &  &  &  &  &  &  &  & &  \\ \cline{2-5}\cline{7-10}
     & \multirow{2}{*}{1/6} & \multirow{2}{*}{0} & \multirow{2}{*}{2} & \multirow{2}{*}{0} &  & \multirow{2}{*}{1/6} & \multirow{2}{*}{0} & \multirow{2}{*}{2} & \multirow{2}{*}{0} &  \\ 
     &  &  &  &  &  &  &  &  & &  \\ \cline{2-5}\cline{7-10}
     & \multirow{2}{*}{1/6} & \multirow{2}{*}{0} & \multirow{2}{*}{0} & \multirow{2}{*}{2} &  & \multirow{2}{*}{1/6} & \multirow{2}{*}{0} & \multirow{2}{*}{0} & \multirow{2}{*}{2} &  \\ 
     &  &  &  &  &  &  &  &  & &  \\ \hline
    \multicolumn{11}{|c|}{\rule[-3pt]{0pt}{12pt}Defender's Payoff = 1/2} \\ \hline\hline
    \multirow{12}{*}{3} & \multirow{3}{*}{2/5} & \multirow{3}{*}{1} & \multirow{3}{*}{1} & \multirow{3}{*}{1} & \multirow{12}{*}{2} & \multirow{2}{*}{2/15} & \multirow{2}{*}{2} & \multirow{2}{*}{0} & \multirow{2}{*}{0} & \multirow{12}{*}{\rd{6.13}} \\ 
     &  &  &  &  &  &  &  &  & &  \\ \cline{7-10}
     &  &  &  &  &  & \multirow{2}{*}{2/15} & \multirow{2}{*}{0} & \multirow{2}{*}{2} & \multirow{2}{*}{0} &  \\ \cline{2-5}
     & \multirow{3}{*}{1/5} & \multirow{3}{*}{2} & \multirow{3}{*}{0} & \multirow{3}{*}{1} &  &  &  &  & &  \\ \cline{7-10}
     &  &  &  &  &  & \multirow{2}{*}{2/15} & \multirow{2}{*}{0} & \multirow{2}{*}{0} & \multirow{2}{*}{2} &  \\ 
     &  &  &  &  &  &  &  &  & &  \\ \cline{2-5}\cline{7-10}
     & \multirow{3}{*}{1/5} & \multirow{3}{*}{1} & \multirow{3}{*}{2} & \multirow{3}{*}{0} &  & \multirow{2}{*}{1/5} & \multirow{2}{*}{1} & \multirow{2}{*}{1} & \multirow{2}{*}{0} &  \\ 
     &  &  &  &  &  &  &  &  & &  \\ \cline{7-10}
     &  &  &  &  &  & \multirow{2}{*}{1/5} & \multirow{2}{*}{1} & \multirow{2}{*}{0} & \multirow{2}{*}{1} &  \\ \cline{2-5}
     & \multirow{3}{*}{1/5} & \multirow{3}{*}{0} & \multirow{3}{*}{1} & \multirow{3}{*}{2} &  &  &  &  & &  \\ \cline{7-10}
     &  &  &  &  &  & \multirow{2}{*}{1/5} & \multirow{2}{*}{0} & \multirow{2}{*}{1} & \multirow{2}{*}{1} &  \\ 
     &  &  &  &  &  &  &  &  & &  \\ \hline
    \multicolumn{11}{|c|}{\rule[-3pt]{0pt}{12pt}Defender's Payoff = 7/5} \\ \hline\hline
    \multirow{12}{*}{2} & \multirow{2}{*}{17/72} & \multirow{2}{*}{2} & \multirow{2}{*}{0} & \multirow{2}{*}{0} & \multirow{12}{*}{3} & \multirow{4}{*}{1/3} & \multirow{4}{*}{2} & \multirow{4}{*}{1} & \multirow{4}{*}{0} & \multirow{12}{*}{\rd{6.01}} \\ 
     &  &  &  &  &  &  &  &  & &  \\ \cline{2-5}
     & \multirow{2}{*}{17/72} & \multirow{2}{*}{0} & \multirow{2}{*}{2} & \multirow{2}{*}{0} &  &  &  &  & &  \\ 
     &  &  &  &  &  &  &  &  & &  \\ \cline{2-5}\cline{7-10}
     & \multirow{2}{*}{17/72} & \multirow{2}{*}{0} & \multirow{2}{*}{0} & \multirow{2}{*}{2} &  & \multirow{4}{*}{1/3} & \multirow{4}{*}{1} & \multirow{4}{*}{0} & \multirow{4}{*}{2} &  \\ 
     &  &  &  &  &  &  &  &  & &  \\ \cline{2-5}
     & \multirow{2}{*}{7/72} & \multirow{2}{*}{1} & \multirow{2}{*}{1} & \multirow{2}{*}{0} &  &  &  &  & &  \\ 
     &  &  &  &  &  &  &  &  & &  \\ \cline{2-5}\cline{7-10}
     & \multirow{2}{*}{7/72} & \multirow{2}{*}{1} & \multirow{2}{*}{0} & \multirow{2}{*}{1} &  & \multirow{4}{*}{1/3} & \multirow{4}{*}{0} & \multirow{4}{*}{2} & \multirow{4}{*}{1} &  \\ 
     &  &  &  &  &  &  &  &  & &  \\ \cline{2-5}
     & \multirow{2}{*}{7/72} & \multirow{2}{*}{0} & \multirow{2}{*}{1} & \multirow{2}{*}{1} &  &  &  &  & &  \\ 
     &  &  &  &  &  &  &  &  & &  \\ \hline
    \multicolumn{11}{|c|}{\rule[-3pt]{0pt}{12pt}Defender's Payoff = $-1/3$} \\ \hline\hline
    \multirow{12}{*}{2} & \multirow{4}{*}{1/3} & \multirow{4}{*}{2} & \multirow{4}{*}{0} & \multirow{4}{*}{0} & \multirow{12}{*}{4} & \multirow{2}{*}{1/6} & \multirow{2}{*}{3} & \multirow{2}{*}{1} & \multirow{2}{*}{0} & \multirow{12}{*}{\rd{19.86}} \\ 
     &  &  &  &  &  &  &  &  & &  \\ \cline{7-10}
     &  &  &  &  &  & \multirow{2}{*}{1/6} & \multirow{2}{*}{0} & \multirow{2}{*}{3} & \multirow{2}{*}{1} &  \\ 
     &  &  &  &  &  &  &  &  & &  \\ \cline{2-5}\cline{7-10}
     & \multirow{4}{*}{1/3} & \multirow{4}{*}{0} & \multirow{4}{*}{2} & \multirow{4}{*}{0} &  & \multirow{2}{*}{1/6} & \multirow{2}{*}{1} & \multirow{2}{*}{0} & \multirow{2}{*}{3} &  \\ 
     &  &  &  &  &  &  &  &  & &  \\ \cline{7-10}
     &  &  &  &  &  & \multirow{2}{*}{1/6} & \multirow{2}{*}{2} & \multirow{2}{*}{1} & \multirow{2}{*}{1} &  \\ 
     &  &  &  &  &  &  &  &  & &  \\ \cline{2-5}\cline{7-10}
     & \multirow{4}{*}{1/3} & \multirow{4}{*}{0} & \multirow{4}{*}{0} & \multirow{4}{*}{2} &  & \multirow{2}{*}{1/6} & \multirow{2}{*}{1} & \multirow{2}{*}{2} & \multirow{2}{*}{1} &  \\ 
     &  &  &  &  &  &  &  &  & &  \\ \cline{7-10}
     &  &  &  &  &  & \multirow{2}{*}{1/6} & \multirow{2}{*}{1} & \multirow{2}{*}{1} & \multirow{2}{*}{2} &  \\ 
     &  &  &  &  &  &  &  &  & &  \\ \hline
    \multicolumn{11}{|c|}{\rule[-3pt]{0pt}{12pt}Defender's Payoff = -1} \\ \hline\hline
    \multirow{12}{*}{4} & \multirow{2}{*}{2/15} & \multirow{2}{*}{2} & \multirow{2}{*}{2} & \multirow{2}{*}{0} & \multirow{12}{*}{2} & \multirow{2}{*}{2/15} & \multirow{2}{*}{2} & \multirow{2}{*}{0} & \multirow{2}{*}{0} & \multirow{12}{*}{\rd{20.55}} \\ 
     &  &  &  &  &  &  &  &  & &  \\ \cline{2-5}\cline{7-10}
     & \multirow{2}{*}{2/15} & \multirow{2}{*}{2} & \multirow{2}{*}{0} & \multirow{2}{*}{2} &  & \multirow{2}{*}{2/15} & \multirow{2}{*}{0} & \multirow{2}{*}{2} & \multirow{2}{*}{0} &  \\ 
     &  &  &  &  &  &  &  &  & &  \\ \cline{2-5}\cline{7-10}
     & \multirow{2}{*}{2/15} & \multirow{2}{*}{0} & \multirow{2}{*}{2} & \multirow{2}{*}{2} &  & \multirow{2}{*}{2/15} & \multirow{2}{*}{0} & \multirow{2}{*}{0} & \multirow{2}{*}{2} &  \\ 
     &  &  &  &  &  &  &  &  & &  \\ \cline{2-5}\cline{7-10}
     & \multirow{2}{*}{1/5} & \multirow{2}{*}{1} & \multirow{2}{*}{1} & \multirow{2}{*}{2} &  & \multirow{2}{*}{1/5} & \multirow{2}{*}{1} & \multirow{2}{*}{1} & \multirow{2}{*}{0} &  \\ 
     &  &  &  &  &  &  &  &  & &  \\ \cline{2-5}\cline{7-10}
     & \multirow{2}{*}{1/5} & \multirow{2}{*}{1} & \multirow{2}{*}{2} & \multirow{2}{*}{1} &  & \multirow{2}{*}{1/5} & \multirow{2}{*}{1} & \multirow{2}{*}{0} & \multirow{2}{*}{1} &  \\ 
     &  &  &  &  &  &  &  &  & &  \\ \cline{2-5}\cline{7-10}
     & \multirow{2}{*}{1/5} & \multirow{2}{*}{2} & \multirow{2}{*}{1} & \multirow{2}{*}{1} &  & \multirow{2}{*}{1/5} & \multirow{2}{*}{0} & \multirow{2}{*}{1} & \multirow{2}{*}{1} &  \\ 
     &  &  &  &  &  &  &  &  & &  \\ \hline
    \multicolumn{11}{|c|}{\rule[-3pt]{0pt}{12pt}Defender's Payoff = 29/15} \\ \hline
    \end{tabular}
  \vspace{0.5000em}
\end{table}

\begin{table}[htbp]
\small
  \centering
  \caption{Equilibrium strategies from semidefinite relaxation of the cyclic Blotto game with $N=4$, $E_1=E_2=3$.
  In the table, columns $p_x$ and $p_y$ denote the probabilities assigned to each listed pure strategy of the defender and attacker in the equilibrium.}
  \label{tab:cyclic_blotto_four}
  \setlength{\tabcolsep}{2pt}
  \begin{tabular}{|c|c|c|c|c|c||c|c|c|c|c|c|c|}
    \hline
    \multicolumn{6}{|c||}{Defender Strategies} & \multicolumn{6}{c|}{Attacker Strategies} & \multirow{2}{*}{\shortstack{Time (s)}} \\ \cline{1-12}
    $E_1$ & $p_x$ & Loc 1 & Loc 2 & Loc 3 & Loc 4 
        & $E_2$ & $p_y$ & Loc 1 & Loc 2 & Loc 3 & Loc 4 &  \\ \hline
    \multirow{8}{*}{3} 
     & 0.1000 & 1 & 0 & 2 & 0 
     & \multirow{8}{*}{3} & 0.1500 & 2 & 1 & 0 & 0 & \multirow{8}{*}{\rd{419.19}} \\ \cline{2-6}\cline{8-12}
     & 0.1000 & 2 & 0 & 1 & 0 
     & & 0.1500 & 0 & 2 & 1 & 0 &  \\ \cline{2-6}\cline{8-12}
     & 0.1000 & 0 & 2 & 0 & 1 
     & & 0.1500 & 0 & 0 & 2 & 1 &  \\ \cline{2-6}\cline{8-12}
     & 0.1000 & 0 & 1 & 0 & 2 
     & & 0.1500 & 1 & 0 & 0 & 2 &  \\ \cline{2-6}\cline{8-12}
     & 0.1500 & 1 & 1 & 0 & 1 
     & & 0.1000 & 1 & 1 & 0 & 1 &  \\ \cline{2-6}\cline{8-12}
     & 0.1500 & 1 & 0 & 1 & 1 
     & & 0.1000 & 1 & 0 & 1 & 1 &  \\ \cline{2-6}\cline{8-12}
     & 0.1500 & 1 & 1 & 1 & 0 
     & & 0.1000 & 1 & 1 & 1 & 0 &  \\ \cline{2-6}\cline{8-12}
     & 0.1500 & 0 & 1 & 1 & 1 
     & & 0.1000 & 0 & 1 & 1 & 1 &  \\ \hline
     \multicolumn{13}{|c|}{Defender's Expected Payoff = 0.5500} \\\hline
  \end{tabular}
  \vspace{-0.3000cm}
\end{table}
\end{example}

% \clearpage
\section{Conclusions}

This paper introduces a general framework for solving a class of max–min bilinear optimization problems over CP  cones, with
applications to zero-sum bi-quadratic games. Our key innovation is the COP–CP reformulation, which transforms the original nested optimization structure into a single-stage linear conic program over the product of COP and CP cones. This reformulation enables us to apply a systematic semidefinite approximation approach, leveraging the moment–SOS hierarchy and the powerful FTC to certify global optimality and extract optimizers.

We provide both asymptotic and finite convergence guarantees for the proposed semidefinite relaxations and demonstrate their applicability on representative numerical examples. 
In particular, we highlight an application to a cyclic Colonel Blotto game, showing how the proposed method can exactly recover the equilibrium strategies and payoff when the FTC is satisfied.

There are several interesting directions for future research. One important avenue is the quantitative characterization of the relaxation gap. Although we provide conditions under which the semidefinite approximation is tight, a deeper understanding of when it is exact and how large the approximation error can be is an open and important question.
Another direction pertains to scalability and computational efficiency. Our numerical study indicates that solving the semidefinite approximations becomes increasingly time-consuming as the problem dimension grows, a common challenge for hierarchy-based relaxation methods. Since the cyclic Blotto game is used as a demonstration of the general method rather than the primary focus of this work, we adopt standard solvers and do not design tailored algorithms for this specific instance. Future research may explore ways to exploit problem structure or develop decomposition techniques to improve scalability for large-scale applications.
Finally, extending the COP–CP reformulation to broader settings, such as problems involving tensor optimization or multi-agent equilibrium computation, may further expand the scope and impact of the proposed framework.

\subsection*{Declaration on the Use of AI}  During the preparation of this work, we used ChatGPT 6 Astra to assist with language modification and numerical calculations. We bear full responsibility for the clarity and explainability of the results, as well as the accuracy, integrity, and originality of the submitted work.

\subsection*{Acknowledgement}
The authors would like to thank Jean Bernard Lasserre and Jiawang Nie for fruitful discussions. 
The second author is partially supported by the Hong Kong Research Grants Council HKBU-15303423 and NSFC Young Scientists Fund, grant number 12301407.

\appendix

\section{Proofs of Statements in Section \ref{sec:formulation}}\label{sc:prfsc3}

\begin{proof}[Proof of Proposition~\ref{prop:maxmin_duality}]
It is straightforward to verify that \eqref{eq:lp_max} and \eqref{eq:lp_min} are {Lagrangian dual problems}, with $\b{w}$ and $Y$ as
multipliers for the equality constraints in \eqref{eq:lp_max}  and $W$
as the copositive dual slack. %

Since \eqref{eq:inner_min} and \eqref{eq:inner_min_dual} are a primal–dual pair, weak duality implies $\mcal{W}^*_{\max}\le \mcal{Z}^*_{\max\min}$ after combining with the outer maximization in \eqref{eqn:conic}. By an analogous weak-duality argument for \eqref{eqn:conic_minmax}, we obtain $\mcal{W}^*_{\min}\ge \mcal{Z}^*_{\min\max}$. Moreover, the standard max--min inequality gives %
$\mcal{Z}^*_{\max\min}\le \mcal{Z}^*_{\min\max}$. 
\end{proof}

\begin{proof}[Proof of Lemma~\ref{lem:z=w}]
We first show $\mcal{Z}_{\max\min}^* = \mcal{W}_{\max}^*$ under Assumption \ref{as:co}(a). The idea is to establish a strong duality between the inner minimization problem~(\ref{eq:inner_min}) and its dual problem \eqref{eq:inner_min_dual} by constructing a Slater point for the dual problem \eqref{eq:inner_min_dual}. Fix $X\in\mcal K_{cp}^x$. Given $z_1\ddd z_J$ satisfying Assumption \ref{as:co}(a), so that $\sum_{j=1}^J z_j B_j \succ_{co} 0$, set $\hat z_j := -Kz_j$ and $Z := \mcal{Q}(X) - \sum_{j=1}^J \hat z_j B_j = \mcal{Q}(X) + K\sum_{j=1}^J z_j B_j$. By strict copositivity of $\sum_{j=1}^Jz_jB_j$ and
compactness of $\Delta^m$, we have $Z\succ_{co}0$
for all sufficiently large $K>0$. %
Hence, $(\hat z_1\ddd \hat z_J, Z)$ is a Slater point for \eqref{eq:inner_min_dual}. Since $\mcal K_{cp}^y$ is nonempty, strong duality therefore holds between \eqref{eq:inner_min} and \eqref{eq:inner_min_dual} for every $X \in \mcal K_{cp}^x$. Therefore, we have $\mcal{Z}_{\max\min}^* = \mcal{W}_{\max}^*$. 

An entirely symmetric argument proves $\mcal{Z}_{\min\max}^* = \mcal{W}_{\min}^*$ under Assumption \ref{as:co}(b). Fix $Y\in\mcal K_{cp}^y$, and take $w_1\ddd w_I$ such that $\sum_{i=1}^I w_i A_i \succ_{co} 0$. Define $\hat w_i := Kw_i$ and $W := \sum_{i=1}^I \hat w_i A_i - \mcal{Q}^{\star}(Y) = K\sum_{i=1}^I w_i A_i - \mcal{Q}^{\star}(Y)$. For  all sufficiently large $K>0$, $W\succ_{co} 0$, and hence $(\hat w_1\ddd \hat w_I, W)$ is a Slater point for the dual COP of the inner maximization in \eqref{eqn:conic_minmax}, implying $\mcal{Z}_{\min\max}^* = \mcal{W}_{\min}^*$. 
\end{proof}

\section{Proofs of Statements in Section \ref{sec:conv_analy}}\label{sc:prfsc4}

\subsection{Proofs related to asymptotic convergence}
In this subsection, we mainly prove the asymptotic convergence result, i.e., Theorem~\ref{tm:asym_conv}, and some related results.
 The following lemma follows from the moment-compactness argument
in \citep[Lemma~8.1 and the proof of Theorem~5.1]
{laraki2012semidefinite} and is useful for proving Theorem~\ref{tm:asym_conv}.
\begin{lemma}{\citep{laraki2012semidefinite}}\label{lem:infi_tms}
 Let $\{\b{\xi}^{(l)}\}_{l=1,2,\ldots}$ be a sequence of tms such that $\b{\xi}^{(l)}_0 = 1$ and $\b{\xi}^{(l)} \in \mathscr{S}[\Dt^n]_{2l}$ for every $l = 1,2,\ldots$.
Then, there exists a subsequence $l_i$ and a probability measure $\mu$ supported on $\Dt^n$ such that 
\[ \lim_{i\to\infty} \b{\xi}^{(l_i)}_{\alpha} = \int_{\Dt^n} \bx^{\alpha} d\mu,\quad \forall\alpha\in \N^n.  \]
\end{lemma}
\begin{proof}[Proof of Theorem~\ref{tm:asym_conv}]
We prove part~\textup{(i)}; part~\textup{(ii)} follows by a
symmetric argument.

To show the asymptotic convergence, it suffices to show the
existence of constants $\mcal{C},\mcal{D}\geq0$ such that, for any
$\varepsilon>0$, the following holds for all sufficiently large
$k_1,k_2$:
\be\label{eq:CWD}
\mcal{W}_{\max}^*-\mcal{C}\varepsilon
\leq
\mcal{W}_{\max}^{(k_1,k_2)}
\leq
\mcal{W}_{\max}^*+\mcal{D}\varepsilon.
\ee

First, we establish the left-hand inequality in \eqref{eq:CWD}.
Let $(X,\b{z},Z)^*$ be an optimal solution of
(\ref{eq:lp_max}). Then
\[\left\{
\begin{array}{ccl}
\mcal{W}_{\max}^*
=&\displaystyle\max_{\b{z},Z}
&\b{c}^{\tp}\b{z}\\
&\st&
\displaystyle
\mcal{Q}(X^*)-\sum_{j=1}^Jz_jB_j=Z,\quad
\psi_m(Z)\in\mcal{P}[\Delta^m]_2^{\hom}.
\end{array}\right.
\]

By Assumption~\ref{as:co}(a), there
exists $\b{z}^{\circ}\in\re^J$ such that
\be\label{eq:zo_pos}
\psi_m\left(
\sum_{j=1}^Jz_j^{\circ}B_j
\right)-1
\in\mcal{P}[\Delta^m]_2.
\ee
\textcolor{black}{Here,
$\mcal P[\Delta^m]_2
:=
\left\{
p\in\re[\b{y}]_2:
p(\b{y})\ge0\ \text{for all }\b{y}\in\Delta^m
\right\}$
which allows constant and linear terms.}
For any fixed $\varepsilon>0$, set
\be\label{eq:z_eps}
\b{z}_{\varepsilon}
\coloneqq
\b{z}^*-\varepsilon\b{z}^{\circ},
\quad
Z_{\varepsilon}
\coloneqq
Z^*+\varepsilon\sum_{j=1}^Jz_j^{\circ}B_j.
\ee
Then
\[
\mcal{Q}(X^*)
-\sum_{j=1}^Jz_{\varepsilon,j}B_j
=
Z_{\varepsilon}.
\]
Moreover, by \eqref{eq:zo_pos},
$\psi_m(Z_{\varepsilon})$ is strictly positive on $\Delta^m$.
By Putinar's Positivstellensatz, there exists
$k_2^{\varepsilon}\in\N_+$ such that
\[
\psi_m(Z_{\varepsilon})
\in\iq[\Delta^m]_{2k_2},
\quad\forall\,k_2\geq k_2^{\varepsilon}.
\]
Since
$\varphi_n(X^*)\in\mcal{R}[\Delta^n]_2^{\hom}$, for every
$k_1\in\N_+$, 
there exists
$\b{\xi}^*\in\mathscr{S}[\Delta^n]_{2k_1}$ such that
\[
\left.\b{\xi}^*\right|_2^{\hom}
=
\varphi_n(X^*).
\]
Thus, for all $k_1\geq1$ and $k_2\geq k_2^\varepsilon$, the tuple
$(X^*,\b{z}_{\varepsilon},Z_{\varepsilon},\b{\xi}^*)$ is feasible
for (\ref{eq:lp_max_relx}).
Let
\[
\mcal{C}
\coloneqq
\Vert\b{z}^{\circ}\Vert\Vert\b{c}\Vert.
\]
Then $\mcal{C}$ is independent of $\varepsilon$, and
\[
\mcal{W}_{\max}^*
-\b{c}^{\tp}\b{z}_{\varepsilon}
=
\b{c}^{\tp}
(\b{z}^*-\b{z}_{\varepsilon})
\leq
\varepsilon
\Vert\b{z}^{\circ}\Vert\Vert\b{c}\Vert
=
\mcal{C}\varepsilon.
\]
Hence, for all $k_1\geq1$ and $k_2\geq k_2^\varepsilon$,
\be\label{eq:W-C}
\mcal{W}_{\max}^*-\mcal{C}\varepsilon
\leq
\b{c}^{\tp}\b{z}_{\varepsilon}
\leq
\mcal{W}_{\max}^{(k_1,k_2)}.
\ee

Next, we prove the right-hand inequality in (\ref{eq:CWD}).
First, we show that Assumption~\ref{as:co}(b) implies that there
exist $\mcal{N}\geq0$ and $k_1^{\circ}\in\N_+$ such that
$\xi_0\leq\mcal{N}$ for every tuple
$(X,\b{z},Z,\b{\xi})$ feasible for (\ref{eq:lp_max_relx}) with
$k_1\geq k_1^{\circ}$.
Let $\b{w}^{\circ}=(w_1^{\circ},\ldots,w_I^{\circ})$ be such that
\[
W^{\circ}\coloneqq\sum_{i=1}^I w_i^{\circ}A_i\succ_{\rm co}0.
\]
WLOG, we assume that $\psi_n(W^{\circ})-1$ is strictly positive on $\Dt^n$.
By Putinar's Positivstellensatz, there exists $k_1^{\circ}$ such
that $\psi_n(W^{\circ})-1\in\iq[\Dt^n]_{2k_1}$ for all
$k_1\geq k_1^{\circ}$. Hence,
\[
0\leq L_{\xi}\big(\psi_n(W^{\circ})-1\big)
=(\b{w}^{\circ})^{\tp}\b{b}-\xi_0.
\]
Therefore,
$0\le\xi_0\le\mcal N:=(\b{w}^\circ)^\tp\b{b}.$
Moreover, since $X^*$ is feasible for \eqref{eq:lp_max},
$\mcal N=\langle W^\circ,X^*\rangle\ge0$.

We show by contradiction that the right-hand inequality in
(\ref{eq:CWD}) holds with
\be\label{eq:mcalD}
\mcal{D}\coloneqq
\Vert\b{c}\Vert\Vert\b{z}^{\circ}\Vert+1.
\ee
Suppose otherwise. Then there exist a sequence of coupled
relaxation orders $(k_1^{(l)},k_2^{(l)})$ and a constant
$\varepsilon_0>0$ such that
\[
\min\{k_1^{(l)},k_2^{(l)}\}\to\infty
\quad\text{as }l\to\infty
\]
and
\be\label{eq:contra}
\mcal{W}_{\max}^{(k_1^{(l)},k_2^{(l)})}
>
\mcal{W}_{\max}^*
+\mcal{D}\varepsilon_0,
\quad \forall\,l=1,2,\ldots.
\ee
For each $l$, choose a tuple
$(\tilde{X},\b{\tilde{z}},\tilde{Z},
\b{\tilde{\xi}})^{(k_1^{(l)},k_2^{(l)})}$
that is feasible for (\ref{eq:lp_max_relx}) with relaxation orders
$(k_1^{(l)},k_2^{(l)})$ and satisfies
\be\label{eq:ctz>=W+D}
\b{c}^{\tp}
\b{\tilde{z}}^{(k_1^{(l)},k_2^{(l)})} >
\mcal{W}_{\max}^*
+\mcal{D}\varepsilon_0.
\ee

Since
$\b{\tilde{\xi}}_0^{(k_1^{(l)},k_2^{(l)})}\le\mcal{N}$
for all sufficiently large $l$, by taking a subsequence if
necessary, we may assume that
\[
\lim_{l\to\infty}
\b{\tilde{\xi}}_0^{(k_1^{(l)},k_2^{(l)})}
=\mcal{M}
\]
for some $0\le\mcal{M}\le\mcal{N}$. There are two cases:
\begin{itemize}[leftmargin=1.5cm]
\item[Case 1:]
$\mcal{M}=0$, and infinitely many terms of
$\b{\tilde{\xi}}^{(k_1^{(l)},k_2^{(l)})}$ satisfy
$\b{\tilde{\xi}}_0^{(k_1^{(l)},k_2^{(l)})}=0:$
By taking this subsequence without relabeling, we assume that the
above holds for all $l$. Let
\[
\b{\eta}^{(l)}
\coloneqq
\b{\tilde{\xi}}^{(k_1^{(l)},k_2^{(l)})}.
\]
Since $\b{\eta}^{(l)}\in\mathscr{S}[\Dt^n]_{2k_1^{(l)}}$ and
$\eta_0^{(l)}=0$, the moment and localizing constraints imply that
$\b{\eta}^{(l)}$ is the zero tms. In this case, we let $\mu$ be the
zero measure on $\Dt^n$.

\item[Case 2:]
$\mcal{M}=0$, but only finitely many terms have
$\b{\tilde{\xi}}_0^{(k_1^{(l)},k_2^{(l)})}=0$, or
$\mcal{M}>0$:
In this case, by removing finitely many terms, we may assume that
$\b{\tilde{\xi}}_0^{(k_1^{(l)},k_2^{(l)})}>0$ for all $l$.
Let
\[
\b{\eta}^{(l)}
\coloneqq
{\b{\tilde{\xi}}^{(k_1^{(l)},k_2^{(l)})}}/
{\b{\tilde{\xi}}_0^{(k_1^{(l)},k_2^{(l)})}}.
\]
Then $\eta_0^{(l)}=1$ and
$\b{\eta}^{(l)}\in
\mathscr{S}[\Dt^n]_{2k_1^{(l)}}$.
Since $k_1^{(l)}\to\infty$, we may pass to a further
subsequence, without relabeling, such that
$k_1^{(l)}\ge l$.
The truncations
$\bigl(\eta_\alpha^{(l)}\bigr)_{|\alpha|\le2l}$
therefore satisfy the hypotheses of Lemma~\ref{lem:infi_tms}. Passing to a further subsequence, there exists a probability measure $\mu$
supported on $\Dt^n$ such that
\[
\lim_{l\to\infty}\eta_{\alpha}^{(l)}
=
\int_{\Dt^n}\bx^\alpha d\mu,
\quad\forall\alpha\in\N^n.
\]
\end{itemize}
In both cases, we have
\[
\lim_{l\to\infty}
\tilde{\xi}_{\alpha}^{(k_1^{(l)},k_2^{(l)})}
=
\mcal{M}\int_{\Dt^n}\bx^\alpha d\mu,
\quad\forall\alpha\in\N^n.
\]
Let
\[
\hat{X}
\coloneqq
\mcal{M}\cdot
\varphi_n^{-1}
\left(
\int_{\Dt^n}[\bx]_2^{\hom}d\mu
\right).
\]
By the definition of $\mcal{R}[\Delta^n]_2^{\hom}$, we have
$\varphi_n(\hat{X})
\in
\mcal{R}[\Delta^n]_2^{\hom}.$
Since
\[
\varphi_n\left(
\tilde{X}^{(k_1^{(l)},k_2^{(l)})}
\right)
=
\left.
\b{\tilde{\xi}}^{(k_1^{(l)},k_2^{(l)})}
\right|_2^{\hom},
\]
and $\varphi_n$ is a linear isomorphism, it follows that
\[
\lim_{l\to\infty} \tilde{X}^{(k_1^{(l)},k_2^{(l)})}
=\hat{X}.
\]
Thus, for every $i=1,\ldots,I$, it holds
\be\label{eq:AhatX=b}
\langle A_i,\hat{X}\rangle
=
\lim_{l\to\infty}
\left\langle
A_i,\tilde{X}^{(k_1^{(l)},k_2^{(l)})}
\right\rangle
=
b_i.
\ee

Consider the following problem:
\be\label{eq:lp_max_xi}
\left\{
\begin{array}{ccl}
\mcal{W}_{\max}^{\hat{X}}\coloneqq
&\displaystyle\max_{\b{z},Z}
&\b{c}^{\tp}\b{z}\\
&\st&
\displaystyle
\mcal{Q}(\hat{X})
-\sum_{j=1}^Jz_jB_j=Z,\quad
\psi_m(Z)\in\mmc{P}[\Delta^m]_2^{\hom}.
\end{array}
\right.
\ee
We next construct, for every sufficiently large $l$, a feasible point
$(\b{\hat{z}}^{(l)},\hat{Z}^{(l)})$ of \eqref{eq:lp_max_xi}.
Denote
\[
p_{\hat{X}}
\coloneqq
\psi_m(\mcal{Q}(\hat{X}))
\quad\text{and}\quad
p^{(l)}
\coloneqq
\psi_m\left(
\mcal{Q}\left(
\tilde{X}^{(k_1^{(l)},k_2^{(l)})}
\right)
\right).
\]
For every $l$, it follows from the feasibility of
\[
(\tilde{X},\b{\tilde{z}},\tilde{Z},
\b{\tilde{\xi}})^{(k_1^{(l)},k_2^{(l)})}
\]
for (\ref{eq:lp_max_relx}) that
\[
p^{(l)}
-
\psi_m\left(
\sum_{j=1}^J
\tilde{z}_j^{(k_1^{(l)},k_2^{(l)})}B_j
\right)
=
\psi_m\left(
\tilde{Z}^{(k_1^{(l)},k_2^{(l)})}
\right)
\in
\iq[\Delta^m]_{2k_2^{(l)}}.
\]
Since
$\psi_m(\tilde{Z}^{(k_1^{(l)},k_2^{(l)})})$
is a  quadratic form, it is clear that
\be\label{eq:pl-psim>=0}
p^{(l)}
-
\psi_m\left(
\sum_{j=1}^J
\tilde{z}_j^{(k_1^{(l)},k_2^{(l)})}B_j
\right)
\in
\mmc{P}[\Delta^m]_2^{\hom}.
\ee
On the other hand, since
$\tilde{X}^{(k_1^{(l)},k_2^{(l)})}
\to\hat{X}$
entrywise, the coefficients of $p^{(l)}$ converge to those of
$p_{\hat{X}}$. By the compactness of $\Delta^m$, for every
$\varepsilon>0$, there exists $l_\varepsilon\in\N$ such that the following holds for all $l\geq l_\varepsilon$:
\be\label{pX-pl+eps>=0}
p_{\hat{X}}-p^{(l)}+\varepsilon
\in\mmc{P}[\Delta^m]_2.
\ee
By Assumption~\ref{as:co}(a), there exists
$\b{z}^{\circ}$ such that \eqref{eq:zo_pos} holds. %
For every $l\geq l_\varepsilon$, define
\be\label{eq:hatzhatZ}
\b{\hat{z}}^{(l)}
\coloneqq
\b{\tilde{z}}^{(k_1^{(l)},k_2^{(l)})}
-\varepsilon\b{z}^{\circ},
\quad
\hat{Z}^{(l)}
\coloneqq
\mcal{Q}(\hat{X})
-\sum_{j=1}^J
\left(
\tilde{z}_j^{(k_1^{(l)},k_2^{(l)})}
-\varepsilon z_j^{\circ}
\right)B_j.
\ee
We have
\[
\resizebox{\linewidth}{!}{$\displaystyle 
\begin{aligned}
\psi_m(\hat{Z}^{(l)})
={}&
p_{\hat{X}}
-
\psi_m\left(
\sum_{j=1}^J
\left(
\tilde{z}_j^{(k_1^{(l)},k_2^{(l)})}
-\varepsilon z_j^{\circ}
\right)B_j
\right)\\
={}&
\underbrace{
p^{(l)}
-
\psi_m\left(
\sum_{j=1}^J
\tilde{z}_j^{(k_1^{(l)},k_2^{(l)})}B_j
\right)
}_{\in\mmc{P}[\Delta^m]_2^{\hom}
\text{ by }\eqref{eq:pl-psim>=0}}
+
\underbrace{
\left(p_{\hat{X}}-p^{(l)}+\varepsilon\right)
+
\varepsilon\left[
\psi_m\left(
\sum_{j=1}^Jz_j^{\circ}B_j
\right)-1
\right]
}_{\in\mmc{P}[\Delta^m]_2^{\hom}
\text{ by }\eqref{eq:zo_pos}
\text{ and }\eqref{pX-pl+eps>=0}}.
\end{aligned}
$}
\]
This implies that
$(\b{\hat{z}}^{(l)},\hat{Z}^{(l)})$ is feasible for
(\ref{eq:lp_max_xi}). Moreover, since
$\varphi_n(\hat{X})\in\mcal{R}[\Delta^n]_2^{\hom}$ and
\eqref{eq:AhatX=b} holds, we have
\[
\mcal{W}_{\max}^*
\geq
\mcal{W}_{\max}^{\hat{X}}.
\]
Particularly, set
$\varepsilon=\varepsilon_0/2$. 
For every $l\geq l_\varepsilon$, it follows from
\eqref{eq:ctz>=W+D} that
\be\label{eq:W-D}
\begin{aligned}
\mcal{W}_{\max}^{\hat{X}}
\geq
\b{c}^{\tp}\b{\hat{z}}^{(l)} & =
\b{c}^{\tp}
\b{\tilde{z}}^{(k_1^{(l)},k_2^{(l)})}
-\varepsilon\b{c}^{\tp}\b{z}^{\circ} \\ 
& \geq
\mcal{W}_{\max}^*
+\mcal{D}\varepsilon_0
-\varepsilon\b{c}^{\tp}\b{z}^{\circ}
\geq
\mcal{W}_{\max}^*
+(1+\Vert c \Vert\Vert z^{\circ} \Vert/2)\varepsilon_0.
\end{aligned}
\ee
Combining this inequality with
$\mcal W_{\max}^*\ge\mcal W_{\max}^{\widehat X}$
gives the contradiction $\mcal{W}_{\max}^* \ge \mcal{W}_{\max}^*
+(1+\Vert c \Vert\Vert z^{\circ} \Vert/2)\varepsilon_0 > \mcal{W}_{\max}^*$, and this completes the proof.

\end{proof}

\begin{proof}[Proof of Corollary~\ref{cor:asym_conv}]
Under Assumption \ref{as:co}, it follows from Proposition \ref{prop:maxmin_duality} and Lemma \ref{lem:z=w} that $\mcal{W}^*_{\max} = \mcal{Z}^*_{\max\min}\le \mcal{Z}^*_{\min\max} =  \mcal{W}^*_{\min}$. By Theorem \ref{tm:asym_conv}, we therefore have \[
\resizebox{\linewidth}{!}{$\displaystyle \lim_{\min\{k_1,k_2\}\to \infty}\mcal{W}^{(k_1,k_2)}_{\max} = \mcal{W}^*_{\max} = \mcal{Z}^*_{\max\min}\le \mcal{Z}^*_{\min\max} =  \mcal{W}^*_{\min} =  \lim_{\min\{k_1,k_2\}\to \infty}\mcal{W}^{(k_1,k_2)}_{\min}. $}
\]

\end{proof}

\subsection{Proofs related to tightness}

\begin{proof}[Proof of Proposition~\ref{prop:ft}]

We prove the first item since the second follows analogously.
Suppose \eqref{eq:ft_max} holds. 
Then the homogeneous degree-$2$ truncation $\b{\xi}^{(k_1,k_2)}|_{2}^{\hom} = \varphi_n(X^{(k_1, k_2)})$ lies in $\mcal{R}[\Delta^n]^{\hom}_2$, by Proposition~8.2.3 of \cite{nie2023moment}. 
{This implies that there exist $\b{x}^{(1)}\ddd \b{x}^{(r_1)}\in \Delta^n$ and $\lambda_1\ddd \lambda_{r_1}\in \re_+$ such that
\[ \b{\xi}^{(k_1,k_2)}|_{2}^{\hom} = \lambda_1 [\b{x}^{(1)}]_2^{\mathrm{\hom}} 
+\cdots+ 
\lambda_{r_1} [\b{x}^{(r_1)}]_2^{\mathrm{\hom}}.\]
Therefore, by the definition of the linear isomorphism $\varphi_n$, one has
\[ X^{(k_1, k_2)} = \lambda_1\b{x}^{(1)}(\b{x}^{(1)})^T+\cdots + \lambda_{r_1}\b{x}^{(r_1)}(\b{x}^{(r_1)})^T \succcurlyeq_{cp}0. \]}%
Meanwhile, from $\psi_m(Z^{(k_1,k_2)})\in\iq[\Delta^{m}]_{2k_2}$, we know it is nonnegative on $\Delta^{m}$, i.e., $\psi_m(Z^{(k_1,k_2)})\in \mmc{P}[\Delta^{m}]^{\hom}_2$. Therefore, $(X,\b{z},Z)^{(k_1,k_2)}$ is a feasible point of Problem \eqref{eq:lp_max}, implying $\mcal{W}_{\max}^{(k_1,k_2)} \le \mcal{W}_{\max}^*$.
\end{proof}

\begin{proof}[Proof of Theorem~\ref{thm:alleq}]
Under both FTCs, Proposition~\ref{prop:maxmin_duality} and Proposition~\ref{prop:ft} give the chain of inequalities:
\be\label{eq:allequalinproof} \mcal{W}_{\max}^{(k_1,k_2)} \le \mcal{W}_{\max}^* \le \mcal{Z}^*_{\max\min} \le \mcal{Z}^*_{\min\max} \le \mcal{W}_{\min}^* \le \mcal{W}_{\min}^{(k_1,k_2)}. \ee
If there is no duality gap between (\ref{eq:lp_max_relx}) and (\ref{eq:lp_min_relx}), meaning $\mcal{W}_{\max}^{(k_1,k_2)} = \mcal{W}_{\min}^{(k_1,k_2)}$, then all these values coincide, as claimed in (\ref{eq:all_eq}).

{Furthermore, (\ref{eq:ft_max}) guarantees that $(X,\b{z},Z)^{(k_1,k_2)}$ is a feasible point to Problem \eqref{eq:lp_max}, by Proposition~\ref{prop:ft}.
 Thus, (\ref{eq:allequalinproof}) implies that $(X,\b{z},Z)^{(k_1,k_2)}$ is also an optimal solution to Problem \eqref{eq:lp_max}.} 
We therefore have the following
\be\label{eq:inner_min_dual_k1k2}
\left\{
\begin{array}{ccl}
\mcal{W}_{\max}^{(k_1,k_2)} = \mcal{W}_{\max}^*=& \max\limits_{z\in\re^J,Z\in \mathbb{S}^m} & c^{\tp}\b{z} \\
& \st & \displaystyle \mcal{Q}(X^{(k_1,k_2)}) - \sum_{j=1}^J z_j B_j = Z,\quad Z\succcurlyeq_{co} 0.
\end{array}\right.
\ee
The dual problem of (\ref{eq:inner_min_dual_k1k2}) is
\be\label{eq:inner_min_k1k2}
\left\{ \begin{array}{cl}
\min\limits_{Y \in  \mathbb{S}^m} & \lip \mcal{Q} (X^{(k_1,k_2)}), Y \rip\\
\st & \lip B_j, Y \rip = {c}_j, \quad j = 1,...,J, \\
& Y\succcurlyeq_{cp} 0.
\end{array}\right.
\ee
By \citep[Proposition~2.4.1]{ben2001lectures} and the assumption $\mcal{W}_{\max}^{(k_1,k_2)} = \mcal{W}_{\min}^{(k_1,k_2)}$, it holds
\[
\resizebox{\linewidth}{!}{$\displaystyle  0 = b^{\tp}\b{w}^{(k_1,k_2)} - c^{\tp}\b{z}^{(k_1,k_2)} = \lip \varphi_n(X^{(k_1,k_2)}),\psi_n(W^{(k_1,k_2)})\rip + \lip \varphi_m(Y^{(k_1,k_2)}),\psi_m(Z^{(k_1,k_2)})\rip. $}
\]
This forces both $\lip \varphi_n(X^{(k_1,k_2)}),\psi_n(W^{(k_1,k_2)})\rip $ and $ \lip \varphi_m(Y^{(k_1,k_2)}),\psi_m(Z^{(k_1,k_2)})\rip$ equal to zero, since they are nonnegative by the fact that $\mathscr{S}[\Dt^n]_{2k_1} = (\iq[\Delta^{n}]_{2k_1})^*$ and $\mathscr{S}[\Dt^m]_{2k_2} = (\iq[\Delta^{m}]_{2k_2})^*$.
{Therefore, it follows from (\ref{eq:all_eq}), the feasibility of $(X,\b{z},Z)^{(k_1,k_2)}$ for Problem~(\ref{eq:lp_max_relx}), and the feasibility of $(Y,\b{w},W)^{(k_1,k_2)}$ for Problem~(\ref{eq:lp_min_relx}) that
 \[ 
 \begin{aligned}
 \mcal{W}^{(k_1,k_2)}_{\min} & = b^{\tp}\b{w}^{(k_1,k_2)}   = \Big\langle \sum_{i=1}^I {w}^{(k_1,k_2)}_iA_i, X^{(k_1,k_2)} \Big\rangle 
  \\
 & = \Big\langle \mcal{Q}^*(Y^{(k_1,k_2)}) + W^{(k_1,k_2)}, X^{(k_1,k_2)} \Big\rangle =  \Big\langle \mcal{Q}^*(Y^{(k_1,k_2)}) , X^{(k_1,k_2)} \Big\rangle\\
 & = \Big\langle \mcal{Q}(X^{(k_1,k_2)}) , Y^{(k_1,k_2)} \Big\rangle = \mcal{W}^{(k_1,k_2)}_{\max}. 
 \end{aligned}\]
 The above implies that $Y^{(k_1,k_2)}$ minimizes (\ref{eq:inner_min_k1k2}) by the duality of (\ref{eq:inner_min_dual_k1k2}) and (\ref{eq:inner_min_k1k2}).
 That is, it solves the inner minimization problem of (\ref{eqn:conic}), given the outer variable equal to $X^{(k_1,k_2)}$.
 Finally, it follows from $\mcal{W}^{(k_1,k_2)}_{\max} = \mcal{Z}^{*}_{\max\min}$ that $(X^{(k_1,k_2)},Y^{(k_1,k_2)})$ solves (\ref{eqn:conic}),} and we can show that it also solves (\ref{eqn:conic_minmax}) in a similar way.
\end{proof}

\begin{proof}[Proof of Theorem~\ref{tm:finiteconv}]
Let $(X,\b{z},Z)^*$ and $(Y,\b{w},W)^*$ be optimal
solutions satisfying Assumption~\ref{as:inIQ}. Recall that we denote $p_{Z^*}\coloneqq \psi_m(Z^*)$ and $p_{W^*}\coloneqq \psi_n(W^*) $.
By \citep[Proposition~2.4.1]{ben2001lectures} and the assumption that $\mcal{W}^{^*}_{\max} = \mcal{W}^{^*}_{\min}$, we have
\be\label{eq:zero_dual_gap} \begin{aligned}
0 = b^{\tp}w^* - c^{\tp}z^* = \lip \varphi_n(X^*),p_{W^*}\rip + \lip \varphi_m(Y^*),p_{Z^*}\rip.
\end{aligned}
\ee

By Assumption~\ref{as:inIQ}, there exist $\hat{k}_1$ and $\hat{k}_2$ such that $p_{W^*}\in \iq[\Delta^n]_{2\hat{k}_1}$ and $p_{Z^*}\in \iq[\Delta^m]_{2\hat{k}_2}$.
Let $k_1\ge \hat{k}_1$ and $k_2\ge \hat{k}_2$ be sufficient large relaxation orders and $(X,\b{z},Z,\xi)^{(k_1,{k}_2)}$ be an arbitrary maximizer of (\ref{eq:lp_max_relx}).
By \citep[Proposition~2.4.1]{ben2001lectures}, for all $(X,\b{z},Z,\xi)$ feasible for (\ref{eq:lp_max_relx}), it holds
\be\label{eq:cz=bw--} \begin{aligned}
& c^{\tp}z = b^{\tp}w^{*} - \lip \varphi_n(X),p_{W^{*}}\rip - \lip \varphi_m(Y^{*}),\psi_m(Z)\rip.
\end{aligned} \ee
Therefore, $(X,\b{z},Z,\xi)^{(k_1,{k}_2)}$ also maximizes
\be\label{eq:cpmin_relx_re}\left\{
\begin{array}{cl}
\displaystyle\max_{X,\b{z},Z,\xi}   &  b^{\tp}w^{*} - \lip \varphi_n(X),p_{W^{*}}\rip - \lip \varphi_m(Y^{*}),\psi_m(Z)\rip\\
 \st & \lip A_i, X \rip = {b}_i, \quad i = 1,..., I\\
     & \displaystyle \mcal{Q}(X) - \sum_{j=1}^J z_j B_j = Z,\ \varphi_n(X) = \xi|^{\hom}_{2},\\
     & \xi \in \mathscr{S}[\Dt^n]_{2k_1},\ \psi_m(Z)\in\iq[\Delta^{m}]_{2k_2}.
\end{array}\right.
\ee
As $\varphi_n(X^{*})\in\mathcal{R}[\Delta^n]_2^{\hom}$, there exist $\bx^{(1)}\ddd \bx^{(r)}$ (let $r = 0$ if $X^*$ is the zero matrix) such that $\bx^{(i)}\in \Delta^n$ for every $i=1\ddd r$, and 
\[ \varphi_n(X^{*}) = \lmd_1[\bx^{(1)}]_2^{\hom}+\lmd_2[\bx^{(2)}]_2^{\hom}+\ldots + \lmd_r[\bx^{(r)}]_2^{\hom}, \]
for some nonnegative scalars $\lmd_1\ddd \lmd_r$.
Let 
\[\xi^* \,:=\, \lmd_1[\bx^{(1)}]_{2k_1}+\lmd_2[\bx^{(2)}]_{2k_1}+\ldots + \lmd_r[\bx^{(r)}]_{2k_1},\]
Then $(X,\b{z},Z,\xi)^*$ is feasible for (\ref{eq:cpmin_relx_re}), since $\psi_m(Z^*) \in \iq[\Delta^m]_{2{k}_2}$ by Assumption~\ref{as:inIQ}, and $\xi^*\in\mathcal{R}[\Delta^n]_{2k_1}\subseteq \mathscr{S}[\Delta^n]_{2k_1}$.
Since $(X,\b{z},Z,\xi)^{(k_1,{k}_2)}$ maximizes (\ref{eq:cpmin_relx_re}), we further have 
\be\label{eq:W*<Wk1k2} 
\mathcal W_{\max}^{*} = c^{\tp} z^* \le c^{\tp} z^{(k_1,k_2)} = \mathcal W_{\max}^{(k_1,k_2)},
\ee
\be\label{eq:k1k2star=0}
\resizebox{0.92\linewidth}{!}{$\displaystyle  \lip \varphi_n(X^{(k_1,k_2)}),p_{W^{*}}\rip + \lip \varphi_m(Y^{*}),\psi_m(Z^{(k_1,k_2)})\rip \le \lip \varphi_n(X^*),p_{W^{*}}\rip + \lip \varphi_m(Y^{*}),p_{Z^*}\rip = 0.$}
\ee
In the above, (\ref{eq:W*<Wk1k2}) follows from (\ref{eq:cz=bw--}), and the last equation of (\ref{eq:k1k2star=0}) is given by (\ref{eq:zero_dual_gap}).
Because $\xi^{(k_1,k_2)}\in \mathscr{S}[\Dt^n]_{2k_1}$, $p_{W^*} = \psi_n(W^*) \in \iq[\Delta^n]_{2{k}_1}$ implies $\lip \varphi_n(X^{(k_1,k_2)}),p_{W^{*}}\rip\ge 0$.
Similarly, $\lip \varphi_m(Y^{*}),\psi_m(Z^{(k_1,k_2)})\rip\ge 0$ follows from $\varphi_m(Y^{*})\in\mathcal{R}[\Delta^m]_2^{\hom}$ and $\psi_m(Z^{(k_1,k_2)})\in\iq[\Delta^{m}]_{2k_2}$.
Equation (\ref{eq:k1k2star=0}) forces 
\be\label{eq:nodualitygap} \lip \varphi_n(X^{(k_1,k_2)}),p_{W^{*}}\rip = \lip \varphi_m(Y^{*}),\psi_m(Z^{(k_1,k_2)})\rip=0.\ee

We first consider the case that the optimal value of (\ref{eq:minPsi}) is zero. 
If $\xi^{(k_1,{k}_2)}_{\bf 0} = 0$, 
then $\langle {\bf 1}, \xi^{(k_1,{k}_2)} \rangle = 0$, 
where ${\bf 1}$ represents the constant polynomial that equals to $1$.
Thus by \cite[Lemma~5.7]{laurent2009sums}, 
$p\in \ker (M_{k_1}[ \xi^{(k_1,{k}_2)} ])$ holds for every $p\in \re^{\N^{n}_{k_1}}$,
which implies that $\xi^{(k_1,{k}_2)}$ is the zero vector and automatically satisfies (\ref{eq:ft_max}).
Suppose that $\xi^{(k_1,{k}_2)}_{\bf 0} =: \xi_0 > 0$.
Consider the $k_1$th moment relaxation of (\ref{eq:minPsi}):
\be\label{eq:minPsi_mom} \left\{\begin{array}{cl} 
\displaystyle \min_{\xi} & \langle p_{W^*}, \xi \rangle  \\
\st & \xi_{\bf 0} = 1,\ \xi\in \mathscr{S}[\Delta^n]_{2k_1}.
\end{array}\right.
\ee 
Since $p_{W^*} \in \iq[\Delta^n]_{2\hat{k}_1}\subset \iq[\Delta^n]_{2k_1}$, and the latter one is the dual cone of $\mathscr{S}[\Delta^n]_{2k_1}$, the optimal value of (\ref{eq:minPsi_mom}) must be nonnegative.
Note that (\ref{eq:minPsi_mom}) is an outer relaxation problem of (\ref{eq:minPsi}), the minimum value of (\ref{eq:minPsi_mom}) is not bigger than zero.
This forces the minimum value of (\ref{eq:minPsi_mom}) equals zero, i.e., the moment relaxation (\ref{eq:minPsi_mom}) is tight.
So by Assumption~\ref{as:inIQ} and \cite[Theorem~2.2]{nie2013certifying}, all minimizers of (\ref{eq:minPsi_mom}) satisfy (\ref{eq:ft_max}).
Remark that $\lip \varphi_n(X^{(k_1,k_2)}),p_{W^{*}}\rip = \lip p_{W^{*}},\xi^{(k_1,k_2)}\rip$ since $\varphi_n(X^{(k_1,k_2)}) = \xi^{(k_1,k_2)}|^{\hom}_{2}$.
It is clear that the normalized tms $\xi^{(k_1,{k}_2)}/\xi_0$ is feasible for (\ref{eq:minPsi_mom}).
Also, (\ref{eq:nodualitygap}) implies that $\langle p_{W^{*}},\xi^{(k_1,{k}_2)}/\xi_0\rangle=0$ .
This implies $\xi^{(k_1,{k}_2)}/\xi_0$ is a minimizer of (\ref{eq:minPsi_mom}).
Therefore, the FTC (\ref{eq:ft_max}) holds for $\xi^{(k_1,{k}_2)}/\xi_0$, and thus, it holds for $\xi^{(k_1,{k}_2)}$ as well.

Now we consider the case that the minimum value of (\ref{eq:minPsi}) is {strictly positive}.
In this case, $X^*$ must be the zero matrix, and there exists $\epsilon>0$ such that $p_{W^*}(x) - \epsilon \in \iq[\Delta^n]_{2k_1}$ for all $k_1$ big enough.
Thus, for all $(k_1,k_2)$ big enough,
(\ref{eq:nodualitygap}) implies that 
\[ 0\le \lip \xi^{(k_1,k_2)} , \psi_n(W^*)- \epsilon\rip = \lip \xi^{(k_1,k_2)} , \psi_n(W^*)\rip - \epsilon \xi^{(k_1,k_2)}_0 = -\epsilon \xi^{(k_1,k_2)}_0 . \]
That is, $\xi^{(k_1,k_2)}_0 = 0$, which means $\xi^{(k_1,k_2)} = {\bf 0}$ and satisfies the FTC, as we showed previously.

Since the maximizer of (\ref{eq:lp_max_relx}) was arbitrary,  (\ref{eq:ft_max}) holds at every optimal solution for all sufficiently large
relaxation orders. A symmetric argument establishes
 (\ref{eq:ft_min}) for every optimal solution of (\ref{eq:lp_min_relx}).
Finally, by (\ref{eq:W*<Wk1k2}) and its analogous for (\ref{eq:lp_min_relx}), Proposition~\ref{prop:ft}, and the assumption that $\mcal{W}^{^*}_{\max} = \mcal{W}^{^*}_{\min}$, 
we have the following holds for all sufficiently large $k_1,k_2$,
\[
\mathcal W_{\max}^{(k_1,k_2)} = \mathcal W_{\max}^*=\mathcal W_{\min}^* = \mathcal W_{\min}^{(k_1,k_2)}.
\]
Thus, there is no duality gap between (\ref{eq:lp_max_relx}) and (\ref{eq:lp_min_relx}), and all conclusions of Theorem~\ref{thm:alleq} follows.
\end{proof}

\section{Proofs for the statements in Section~\ref{sec:game}}

For the proofs in this section, define the moment maps
\[
\Phi_i(\mu_i):=(\bE_{\mu_i}[\bs_i],\bE_{\mu_i}[\bs_i\bs_i^\tp]),
\qquad \mu_i\in\mcal M_i,\quad i=1,2.
\]

\begin{proof}[Proof of Proposition~\ref{prop:bqg=cpg}]
For any \(\mu_i\in\mcal M_i\), it is clear that 
\(\Phi_i(\mu_i)\in\mcal C_{cp}^i\), and
\[
\bE_{\mu_1\otimes\mu_2}[P(\bs_1,\bs_2)]
=
V\bigl(\Phi_1(\mu_1),\Phi_2(\mu_2)\bigr).
\]
Let \((\mu_1^*,\mu_2^*)\) be a Nash equilibrium of
\eqref{eq:mix_bigame}, and let
$(\bp^*,P^*):=\Phi_1(\mu_1^*),
(\bq^*,Q^*):=\Phi_2(\mu_2^*).$
Under Assumption~\ref{as:0to1}, \citep[Theorem~2.6]{burer2009copositive} gives
\be\label{eq:maxpp=p*P*}
\begin{aligned}
\max_{(\bp,P)\in\mcal C_{cp}^1}
V\bigl((\bp,P),(\bq^*,Q^*)\bigr)
=
\max_{\bs_1\in\mcal S_1}
V\bigl((\bs_1,\bs_1\bs_1^\tp),(\bq^*,Q^*)\bigr)
=
V\bigl((\bp^*,P^*),(\bq^*,Q^*)\bigr).
\end{aligned}
\ee
The last equality follows from the left inequality in
\eqref{eq:mix_bigame}, together with the fact that $V$ is linear in $(\bp,P)$ for fixed $(\bq^*,Q^*)$.
Analogously, we also have
\be\label{eq:inqq=q*q*}
\min_{(\bq,Q)\in\mcal C_{cp}^2}
V\bigl((\bp^*,P^*),(\bq,Q)\bigr)
=V\bigl((\bp^*,P^*),(\bq^*,Q^*)\bigr).
\ee
For every fixed $(\bp,P)\in\mcal C_{cp}^1$, (\ref{eq:maxpp=p*P*}) implies that 
\be\label{eq:label}
\min_{(\bq,Q)\in\mcal C_{cp}^2}V((\bp,P),(\bq,Q))
\le V((\bp,P),(\bq^*,Q^*))
\le V((\bp^*,P^*),(\bq^*,Q^*)).
\ee
So we have
\[ \mcal{V}_1^* \le V((\bp^*,P^*),(\bq^*,Q^*)). \]
On the other hand, it follows from (\ref{eq:inqq=q*q*}) that $(\bq^*,Q^*)$ minimize $V$ given $(\bp^*,P^*)$, thus $V((\bp^*,P^*),(\bq^*,Q^*))$ is attainable for (\ref{eq:cpmaximin}).
Therefore, we know $(\bp^*,P^*)$ is
optimal for \eqref{eq:cpmaximin}.
A symmetric argument shows that $(\bq^*,Q^*)$ is optimal for
\eqref{eq:cpminimax} with the same optimal value. Thus, both optimal values equal
$V((\bp^*,P^*),(\bq^*,Q^*))
=\bE_{\mu_1^*\otimes\mu_2^*}[P(\bs_1,\bs_2)]$.
\end{proof}

\begin{proof}[Proof of Proposition~\ref{prop:gamecpco}]
The duality can be shown by direct verification and we omit the details.

We show the first item that
\(\mcal V_1^*=\mcal U_1^{cop-cp}\)
when \(H^\tp H\succ_{co}0\).
The second item can be similarly showed by symmetry.
Fix \((\bp,P)\in\mcal C_{cp}^1\).
The inner minimization in \eqref{eq:cpmaximin}
is a CP program whose conic dual problem is
\[
\max_{(\rho,\b{r},\b{v},\b{\vartheta}\rd{,\Xi})
      \in\mcal C_{cop}^2(\bp,P)}
\rho+\b{h}^\tp\b{r}
+(\b{h}\bullet\b{h})^\tp\b{v}.
\]
As in the proof of Lemma~\ref{lem:z=w}, we establish strong
duality by constructing a strictly feasible point for this dual problem.
Set
$\rho=-K,
\b{r}=0,
\b{v}=-K\b{e},
\b{\vartheta}=0\rd{,\ \Xi=0}.$
The slack matrix variable in
\(\mcal C_{cop}^2(\bp,P)\) becomes
\[
\begin{bmatrix}
K & \bp^\tp C/2\\
C^\tp\bp/2 & \mcal F(P)+KH^\tp H
\end{bmatrix}
=
\begin{bmatrix}
0 & \bp^\tp C/2\\
C^\tp\bp/2 & \mcal F(P)
\end{bmatrix}
+
K\begin{bmatrix}
1 & 0\\
0 & H^\tp H
\end{bmatrix}.
\]
Since \(H^\tp H\succ_{co}0\), the second block
matrix is strictly copositive. Hence, for all
sufficiently large \(K>0\), this matrix variable
is strictly copositive, yielding a Slater point.
Since \(\mcal C_{cp}^2\) is nonempty, conic strong
duality gives equality of the inner CP and COP
optimal values. Maximizing over
\((\bp,P)\in\mcal C_{cp}^1\) gives
$\mcal V_1^*=\mcal U_1^{cop-cp}.$
\end{proof}

\begin{proof}[Proof of Theorem~\ref{thm:gameNE}]
\textup{(i)}
Since $\mcal S_1,\mcal S_2$ are nonempty and compact
and the payoff is continuous, the mixed-strategy
bi-quadratic game admits a Nash equilibrium
\citep{glicksberg1952further}.
By Proposition~\ref{prop:bqg=cpg}, the first and second moments of
this equilibrium are optimal for \eqref{eq:cpmaximin} and
\eqref{eq:cpminimax}, respectively, and
\[
\mcal V_1^*=\mcal V_2^*.
\]
By Proposition~\ref{prop:gamecpco}, it remains to show
that
$G^\tp G\succ_{co}0$ and 
$H^\tp H\succ_{co}0$.
Let $\b{d}\ge0$ be such that $G\b{d}=0$, and choose
$\bs_1\in\mcal S_1$.
Then $\bs_1+t\b{d}\ge0$ and $G(\bs_1+t\b{d})=\b g$
for every $t\ge0$.
Assumption~\ref{as:0to1} implies
$d_i=0$ for all $i\in\mathbb B_1$.
Consequently,
$\bs_1+t\b{d}\in\mcal S_1$ for every $t\ge0$.
Compactness of $\mcal S_1$ therefore forces
$\b{d}=0$.
It follows that
\[
\b{d}^\tp G^\tp G\b{d}
=
\|G\b{d}\|^2>0,
\qquad
\b{d}\ge0,\quad \b{d}\ne0,
\]
so $G^\tp G\succ_{co}0$.
The same argument shows that $H^\tp H\succ_{co}0$.
Proposition~\ref{prop:gamecpco} then gives
\[
\mcal V_1^*
=
\mcal U_1^{cop-cp}
=
\mcal U_2^{cop-cp}
=
\mcal V_2^*.
\]

\textup{(ii)}
By \citep[Corollary~2.5]{burer2009copositive}, we have
\[
\mcal C_{cp}^i
=
\operatorname{conv}
\left\{
(\bs_i,\bs_i\bs_i^\tp):
\bs_i\in\mcal S_i
\right\}
=
\Phi_i(\mcal M_i),
\qquad i=1,2.
\]
Moreover,
\be\label{eq:Emu1mu2=Vphi1phi2}
\bE_{\mu_1\otimes\mu_2}[P(\bs_1,\bs_2)]
=
V\bigl(\Phi_1(\mu_1),\Phi_2(\mu_2)\bigr).
\ee
The `only if' implication follows from Proposition~\ref{prop:bqg=cpg}.
For the `if' implication, suppose that
$(\bp^*,P^*):=\Phi_1(\mu_1^*)$ and
$(\bq^*,Q^*):=\Phi_2(\mu_2^*)$ are optimal for
\eqref{eq:cpmaximin} and \eqref{eq:cpminimax}, respectively.
By item \textup{(i)},
for all $(\bp,P)\in\mcal C_{cp}^1$ and
$(\bq,Q)\in\mcal C_{cp}^2$, it holds
\[
V((\bp,P),(\bq^*,Q^*))\le \mcal V_1^*=\mcal V_2^*
\le V((\bp^*,P^*),(\bq,Q)).
\]
So it follows from (\ref{eq:Emu1mu2=Vphi1phi2}) that \eqref{eq:mix_bigame} holds, i.e., $(\mu_1^*,\mu_2^*)$ is a
mixed-strategy Nash equilibrium.

\textup{(iii)}
Suppose $(\bp^*,P^*)$ admits a maximizer of \eqref{eq:cpmaxmax}
and $(\bq^*,Q^*)$ admits a minimizer of \eqref{eq:cpminmin}.
Then by weak duality, the following holds for every \((\bp,P)\in\mcal C_{cp}^1\) and
\((\bq,Q)\in\mcal C_{cp}^2\):
\[
V\bigl((\bp,P),(\bq^*,Q^*)\bigr)
\le \mcal U_1^{cop-cp}=\mcal U_2^{cop-cp}
\le
V\bigl((\bp^*,P^*),(\bq,Q)\bigr)
\]
Thus, by letting
\((\bp,P)=(\bp^*,P^*)\) and
\((\bq,Q)=(\bq^*,Q^*)\), we have
\[
V\bigl((\bp^*,P^*),(\bq^*,Q^*)\bigr)=\mcal U_1^{cop-cp}=\mcal U_2^{cop-cp}.
\]
Therefore, by item \textup{(i)}, $(\bp^*,P^*)$ and $(\bq^*,Q^*)$ are optimal solutions for \eqref{eq:cpmaximin} and \eqref{eq:cpminimax}, respectively.
Since $\mcal C_{cp}^i=\Phi_i(\mcal M_i)$, there exist mixed strategies
with these first and second moments, which form a Nash equilibrium
by \textup{(ii)}.
\end{proof}

\bibliographystyle{abbrv}
\bibliography{MaxminCOP_arXiv_2nd}

\end{document}